\documentclass{article}
\usepackage{amsfonts}
\usepackage{amssymb}
\usepackage{amsmath,mathrsfs}
\usepackage{amsthm}
\usepackage{latexsym}
\usepackage{layout}
\usepackage{epsfig}
\usepackage{graphicx}
\usepackage{pdfsync}
\usepackage{cite}

\usepackage[english]{babel}

\theoremstyle{plain}
\begingroup
\newtheorem{theorem}{Theorem}[section]
\newtheorem{lemma}[theorem]{Lemma}
\newtheorem{proposition}[theorem]{Proposition}

\endgroup

\theoremstyle{definition}
\begingroup
\newtheorem{definition}[theorem]{Definition}
\newtheorem{remark}[theorem]{Remark}
\newtheorem{example}[theorem]{Example}
\endgroup

\theoremstyle{remark}

\mathchardef\emptyset="001F

\numberwithin{equation}{section}

\newcommand{\ol}{\overline}

\newcommand{\R}{{\mathbb R}}
\newcommand{\Rn}{{\R}^N}

\newcommand{\Om}{\Omega}

\newcommand{\N}{\mathbb N}
\newcommand{\Z}{\mathbb Z}

\newcommand{\Mdue}{{\mathbb M}^{2{\times}2}_{sym}}
\newcommand{\Mtre}{{\mathbb M}^{3{\times}3}_{sym}}
\newcommand{\Mnn}{{\mathbb M}^{N{\times}N}_{sym}}

\newcommand{\Md}{{\mathbb M}^{2{\times}2}_{D}}

\newcommand{\wtos}{\mathrel{\mathop{\rightharpoonup}\limits^*}}

\newcommand{\e}{\varepsilon}

\newcommand{\be}{\begin{equation}}
\newcommand{\ee}{\end{equation}}
\newcommand{\bes}{\begin{eqnarray}}
\newcommand{\ees}{\end{eqnarray}}

\newcommand \dps{\displaystyle }

\newcommand{\zN}{\mathbb{Z}^N}

\newcommand{\tr}{\hbox{tr}}

\newcommand{\al}{\alpha}

\def\a{\alpha}

\title{$Q$-tensor continuum energies as limits of head-to-tail symmetric spin systems}
\author{Andrea Braides\thanks{Dipartimento di Matematica, Universit\`a di Roma `Tor Vergata' via della Ricerca Scientifica, 00133 Roma, Italy. Email: {\tt braides@mat.uniroma2.it}} \and Marco Cicalese\thanks{Zentrum Mathematik - M7, Technische Universit\"at M\"unchen, Boltzmannstrasse 3, 85748 Garching, Germany. Email: {\tt cicalese@ma.tum.de}} \and Francesco Solombrino\thanks{Zentrum Mathematik - M7, Technische Universit\"at M\"unchen, Boltzmannstrasse 3, 85748 Garching, Germany. Email: {\tt francesco.solombrino@ma.tum.de}}}

\begin{document}

\maketitle

\begin{abstract}
We consider a class of spin-type discrete systems and analyze their continuum limit as the lattice spacing goes to zero. Under standard coerciveness and growth assumptions together with an additional head-to-tail symmetry condition, we observe that this limit can be conveniently written as a functional in the space of $Q$-tensors. We further characterize the limit energy density in several cases (both in $2$ and $3$ dimensions). In the planar case we also develop a second-order theory and we derive gradient or concentration-type models according to the chosen scaling. 
\end{abstract}

\section{Introduction}
The application of $\Gamma$-convergence to the study of discrete systems allows the rigorous 
definition of continuum limits for variational problems starting from point interactions. 
This discrete-to-continuum approach consists in 
first introducing a small geometric parameter $\e$ and defining suitable energies whose domain are functions 
 $\{u_i\}$ parameterized by the nodes of a lattice of lattice spacing $\e$, and then identifying those functions with
suitable continuous interpolations. This allows to embed these energies in classes of functionals that
can be studied by `classical' methods of $\Gamma$-convergence. Examples of applications of this 
approach comprise $u_i$ representing an atomic displacement (in which case the continuous parameter
can be interpreted as a macroscopic displacement) or a spin variable (in which case it may give rise
to energies depending on a macroscopic magnetization). It must be noted that the choice of the relevant
macroscopic variable is a fundamental part of the problem, as well as the scaling of the energies.
Following this approach it has been possible, e.g., to prove compactness and integral representation theorems
for volume and surface integrals \cite{ACG,AC,AAG} that follow the localization arguments of classical results
for continuum energies \cite{GCB,DM}. On the other hand the constraints given by the discrete nature of the 
parameters, often entail interesting features of the limit energies (for example, optimality properties for discrete 
linear elastic composites \cite{BF}, multi-phase limits for next-to-nearest neighbor scalar spin systems \cite{ABC},
`surfactant'-type theories \cite{ACS}, etc.). 

Concerning the continuum limits of spin systems, several results have been recently obtained. In few words,
taking the cubic lattice $\Z^N$ as reference, the set up of the problem is the following: given $\Omega\subset\Rn$, $\e\Z_\e(\Omega):=\e\Z^N\cap\Omega$, $Y\subseteq S^{N-1}$ and denoting by $u:\e\Z_\e(\Omega)\to Y$ the spin field, one is interested in the limit as $\e\to 0$ of energies of the form
\begin{eqnarray}\label{intro:energy}
E_\e(u)=\sum_{\xi\in \zN }\sum_{\alpha\in R_\e^\xi(\Omega)} \e^N
g^\xi_\e(\e\alpha,u(\e\alpha),u(\e\al+\e\xi)),
\end{eqnarray}
with $R_\e^\xi(\Omega):=\{\a\in\Z_\e(\Om):\ \a+\xi\in\Z_\e(\Om)\}$. 
The energy density $g^\xi_\e$ represents the interaction potential between points at distance $\e\xi$ in $\Z_\e(\Omega)$. 
Under the {\it exchange symmetry} condition
\begin{eqnarray*}
g_\e^\xi (\e\alpha,u,v)&=&g_\e^{-\xi}(\e\alpha+\e\xi,v,u)
\end{eqnarray*}
and suitable decay assumptions on the strength of the potentials $g_\e^\xi$ as $|\xi|$ diverges, an integral representation result has been proved in \cite{ACG} asserting that, up to subsequences, 
\begin{equation*}
\Gamma{\mbox-}\lim_\e E_\e(u)=\int_\Omega g(x,u(x))\ dx.
\end{equation*} 
Note that, even if the knowledge of this bulk limit does not describe with enough details some of the features of many spin systems, for which the analysis of higher scalings is needed, nevertheless the characterization of $g$ is a necessary starting point. Concerning the analysis at higher order, it is worth noting that few general abstract results are available (interesting exceptions being the forthcoming paper \cite{AAG}, and the case of periodic interactions \cite{BP}). As a matter of fact, in most of the cases the analysis has to be tailored to the specific features of the discrete system and in particular to the symmetries of its energy functional. Following this general idea, the first problem to face is the definition of a `correct' order parameter which may keep track of the concentration of energy at the desired scale. Sometimes the choice of the order parameter is simply the (weak limit of the) spin field $u$ itself. This happens for the ferromagnetic nearest-neighbor energies of the Ising systems (when $\#Y=2$) discussed in \cite{ABC}, as well as for ternary-type Blume-Emery-Griffith models considered in \cite{ACS} (when $\#Y=3$, or higher). In these cases the $\Gamma$-limits of the energies may be compared with suitable variational smooth interpolations (as in \cite{BLB,BLBL}) or scaling limits in Statistical Mechanics \cite{P}.
In other cases, as for instance for nearest and next-to-nearest ferromagnetic Ising systems discussed in \cite{ABC} and in \cite{BC}, the choice of the parameter is driven by the knowledge of the ground states of the system.

In the present paper we consider discrete energies as in \eqref{intro:energy} satisfying an additional {\em head-to-tail symmetry condition}, namely 
\begin{eqnarray}\label{intro:head-to-tail}
g_\e^{\xi}(\e\alpha,-u,v)=g_\e^\xi (\e\alpha,u,v)=g_\e^{\xi}(\e\alpha,u,-v).
\end{eqnarray}
The condition above, which entails that antipodal vectors may not be distinguished energetically, motivates the choice of the order parameter. Even though ground states may exhibit complex microstructures, the description of the overall properties of the system can be described in terms of the De Gennes  {\em $Q$-tensor} associated to $u$; namely, $Q(u)=u\otimes u$ (see for instance \cite{DeGPro,MotNew}). 

Among the physical models driven by energies satisfying an {\em head-to-tail} symmetry condition a special role is played by nematics. In particular, energetic models belonging to the class we consider here are the so called lattice Maier-Saupe models, firstly introduced by Lebwohl and Lasher in \cite{LebLas} as a simplification of the celebrated mean-field Maier-Saupe model for liquid crystals (see \cite{MaiSau}). The lattice model introduced by Lebwohl and Lasher neglects the interaction between centers of mass of the molecules (these being fixed on a lattice) and penalizes only alignment. Even if it does not reproduce any liquid feature of nematics, this model has proved to be quite a good approximation of the general theory in the regime of high densities and at the same time less demanding from the computational point of view. For these reasons it has been subsequently widely generalized by many authors (some interesting development of the model are presented in \cite{Silvano,Zannoni,Kay}).  

With the choice of the $Q$-tensor order parameter the energy in \eqref{intro:energy} takes the form
\begin{eqnarray*}
F_\e(Q)=\sum_{\xi\in \zN }\sum_{\alpha\in R_\e^\xi(\Omega)} \e^N
f^\xi_\e(\e\alpha,Q(\e\alpha),Q(\e\al+\e\xi)),
\end{eqnarray*}
which underlines that in  \eqref{intro:energy} we can rewrite energies as depending only on $u\otimes u$.
Besides decay assumption on long-range interactions, we suppose that
the potentials $f_{\e}^{\xi}$ satisfy the {\it exchange symmetry} condition:
\begin{eqnarray*}
f_\e^\xi (\e\alpha,Q_1,Q_2)=f_\e^{-\xi}(\e\alpha+\e\xi,Q_2,Q_1),
\end{eqnarray*}
assumption \eqref{intro:head-to-tail} being now expressed by the structure of $Q$. An application of Theorems 3.4 and 5.3 in \cite{ACG} gives the integral-representation and homogenization results stated in Theorems \ref{32-th:disc-cont-Li-mean} and \ref{32-th:homog-Li}. In the particular case when $f_\e^\xi (\e\alpha,Q_1,Q_2)=f^\xi (Q_1,Q_2)$ does not depend on the space variable and $\e$, we obtain an integral-representation formula for the $\Gamma$-limit of $F_\e$ of the type
\begin{equation}\label{intro:fhom}
\Gamma\hbox{-}\lim_\e F_\e(Q)=\int_\Omega f_{\rm hom}(Q(x))\ dx,
\end{equation} 
where $f_{\rm hom}$ is given by an abstract asymptotic homogenization formula (see \eqref{32-eq:homog-form_li}). 
The advantages of this reformulation are several. The first one is that now the limit functional depends on the $Q$-tensor, thus the expected head-to-tail symmetry property of the continuum model is satisfied. Moreover, continuum models involving a $Q$-tensor variable have been the object of intense studies in recent years (see \cite{BalZar} and \cite{ZarLN}) so that the systems we study may be seen as a discrete approximation of some models of this kind. Furthermore, from a technical point of view, the understanding of the algebraic structure of the space of the $Q$-tensors proves to be useful for better characterizing the limit energies in many special cases under different scalings. This is indeed the main object of this paper and to it we devote the last part of the introduction. 

Starting from the abstract formula \eqref{intro:fhom} and inspired by some of the above-mentioned physical models, our first purpose is to characterize $f_{\rm hom}$ for special choices of $g_\e^\xi$. In Section \ref{2d} we give a complete analysis in the case of planar nearest-neighbor interactions; i.e,
when $g_\e^\xi$ is non zero only for $|\xi|=1$, and for such $\xi$ we have $g_\e^\xi(u,v)= f(u,v)$. In particular in Theorem \ref{theorem:nn_general} we prove that $f_{\rm hom}=4\widehat f^{**}$, where $\widehat f^{**}$ is the convex envelope of $\widehat f$ defined by the relaxation formula
\begin{equation}\label{fhat0}
\widehat f(Q):=
\begin{cases}
f(u,v)&\text{if }\displaystyle \frac{u\otimes u+v\otimes v}{2}=Q\neq\frac{1}{2}I\\
\displaystyle \min\bigl\{f(u,v):\ u,\ v\in S^1, \ u\cdot v=0\bigr\} &\text{if }Q=\frac12 I. 
\end{cases}
\end{equation} 
An interesting feature of this formula is that the continuum energy of the macroscopically unordered state $Q=\frac12 I$ is obtained by approximating $Q$ at a microscopic level following an optimization procedure among all the possible pairs of orthogonal vectors $u,v\in S^1$. In the anisotropic case, this can be read as a selection criterion at the micro-scale whenever the minimum in formula \eqref{fhat0} is not trivial. In the homogeneous and isotropic case, instead, that is when $f$ is such that
$$
f(Ru,Rv)=f(u,v)
$$ 
for all $u,v\in S^1$ and all $R\in SO(2)$, the formula for $f_{\rm hom}$ can be further simplified and it only involves the relaxation of a function of the scalar variable $|Q-\frac12 I|$. This radially symmetric energy functional penalizes the distance of a microscopic state to the unordered one. The proof of this result is based on a dual-lattice approach and strongly makes use of a characterization of the set of $Q$-tensors in dimension two, which turns out to be `compatible' with the structure of two-point interactions on a square lattice (see Propositions \ref{K_prop} and \ref{formula_magica}). 

The extension of the arguments of Section \ref{2d} to the three-dimensional case is not straightforward due to the much more complex structure of the space of $Q$-tensors in higher dimensions (see Lemma~\ref{decomp3}). We were able to find a cell formula for $f_{\rm hom}$ only for a particular class of energies. Indeed, our dual-lattice approach fits quite well with energies depending on the set of values that the vector field $u$ takes on the $4$ vertices of each face of a cubic cell, independently of their order. Two-body type potentials giving raise to this special energy structure necessarily involve nearest and next-to-nearest neighbor interactions satisfying the special relations that we consider in Section \ref{3d}.

In the last section we further analyze the planar case by considering different scalings of homogeneous and isotropic
energies whose bulk limit provides little information on the microscopic structure of the ground states. In Theorem \ref{gradterm} we study a class of nearest-neighbor interaction energies and prove that their $\Gamma$-limit is an integral functional whose energy density is proportional to the squared modulus of the $Q$-tensor. Our result contains as a special case the analysis of the well-known  Lebwohl-Lasher model of nematics, in which case we prove that the energy favors a uniform distribution of vector fields at the microscopic scale. Another class of energies is analyzed in Theorem~\ref{thm:oscillating}. Here, as a consequence of the competition between nearest and next-to-nearest interactions, the $\Gamma$-limit, 
while again proportional to the squared modulus of $Q$, favors oscillating microscopic configurations. In both the above cases the limit energy, of the form
$$
\frac{\gamma}{s^2}\int_\Om|\nabla Q(x)|^2\ dx,
$$
can be interpreted as the cost of unit spatial variations of $Q$ on the sub-manifold $|Q-\frac12 I|=\frac{\sqrt{2}}{2} s$ of the space of $Q$-tensors. Being the pre-factor an increasing function of the distance of $Q$ from the set of ordered states, this energy can be interpreted as a measure of the microscopic disorder of the system.  In the analysis done in Section \ref{sec:grad} a crucial role is played by the assumption of simple connectedness of the domain $\Om$:  in that case one may take advantage of the orientability of Sobolev $Q$-tensor fields having constant Frobenius norm
(as clarified in Remark \ref{remark:approx}). In the last subsection of the paper we describe some of the features of head-to-tail symmetric spin systems under scalings allowing for the emergence of topological singularities. To that end we focus on the Lebwohl-Lasher model under a logarithmic type scaling (for the analysis of more general long-range models see Remark \ref{rem:concentration-longrange}). Namely, we consider
$$
E_\e(u)=\frac{1}{|\log \e |}\sum_{|i-j|=1}(1-|(u_i\cdot u_j)|^2)
$$
and prove that in an appropriate topology its $\Gamma$-limit leads to concentration on point singularities. From a microscopical point of view (see Figure \ref{fig_vortex}) this asserts that microscopic ground states look like a finite product of complex maps with half-integer singularities. Note that discrete systems under concentration scalings have been recently studied also in \cite{ACXY}, \cite{ACP}, \cite{AP} and \cite{ADLGP}. 

%As a final comment, we believe that some of the techniques developed in this paper may be useful to study the continuum limit of head-to-tail symmetric spin systems driven by more than two-points interaction energies or in higher dimension, where the characterization of the space of $Q$-tensors . Systems of these type may be interesting as a possible discrete approximation of $Q$-tensor continuum models of nematics (see for instance \cite{DeGPro}).

\section{Notation and Preliminaries}
Throughout the paper $\Om\subset\Rn$ is a bounded open set with Lipschitz boundary. Further hypotheses on $\Om$ will be specified when necessary.
We define $\Z_\e(\Omega)$ as the set of points $i \in \Z^N$ such that $\e i \in \Om$.
The $N$-dimensional reference cube $\bigl[-\frac12, \frac12\bigr)^N$ of $\Rn$ is denoted by $W_N$. The set $\Om_\e$ is then defined as the union over all $i \in \Z^N$, of all the cubes $\e \{i+W_N\}$ such that $\e \{i+W_N\} \subset \subset \Om$. 
In the case $N=2$, which we will be mainly concerned with, we use the shorthand $W$ in place of $W_2$.
The standard norms in euclidean spaces will be always denoted by $|\cdot|$; this holds in particular for the euclidean norm on $\R^N$ as well as for the Frobenius norm on the space of $N\times N$ matrices $\mathbb M^{N\times N}$ and on the subspace of $N\times N$ symmetric matrices $\Mnn$. For these given metrics, $B(x,\rho)$ and $B(Q,\rho)$ will denote the open balls of radius $\rho>0$ centered at $x\in \R^N$, and $Q\in \Mnn$, respectively. 
The symbol $S^{N-1}$ stands as usual for the unit sphere of $\R^N$.

Given two vectors $a$ and $b$ in $\Rn$, the tensor product $a\otimes b$ is the $N\times N$ matrix componentwise defined by $(a\otimes b)_{ij}=a_i b_j$ for all $i, j=1,\dots, N$. Note that, if $c$ and $d$ are also vectors in $\R^N$, 
\begin{equation}\label{prodotti}
(a\cdot c)(b\cdot d)=(a\otimes b):(c\otimes d)\,.
\end{equation}
Here $\cdot$ is the euclidean scalar product, while $:$ is the scalar product between matrices inducing the Frobenius norm. In particular we have
\begin{equation}\label{modul}
|a\otimes b|=|a||b|\,.
\end{equation}
Furthermore, the action of the matrix $a\otimes b$ on a vector $c$ satisfies
\begin{equation}\label{azione}
(a\otimes b)c= (b\cdot c)a.
\end{equation}
We will make often use of the following tensor calculus identity.
\begin{proposition}
Let $u\colon \Om \to S^{N-1}$ be a $C^1$ function. Then
\begin{equation}\label{fattoregiusto}
|\nabla (u\otimes u)(x)|^2= 2|\nabla u(x)|^2
\end{equation}
for all $x\in \Om$.
\end{proposition}
\begin{proof}
Denoting with $e_i$ with $i=1,\dots, N$ the vectors of the canonical basis of $\Rn$ and using \eqref{prodotti} and \eqref{modul}, one has
\begin{eqnarray*}
 |\nabla (u\otimes u)(x)|^2
&=&\lim_{h\to 0}\sum_{i=1}^N\Big|\frac{(u\otimes u)(x+he_i)-(u\otimes u)(x)}{h}\Big|^2
\\
&=&
2\lim_{h\to 0}\sum_{i=1}^N\frac{1-(u(x+he_i)\cdot u(x))^2}{h^2}\\
&=&2\lim_{h\to 0}\sum_{i=1}^N\frac{1-(u(x+he_i)\cdot u(x))}{h^2}[1+(u(x+he_i)\cdot u(x))]
\\
&=&
\lim_{h\to 0}\sum_{i=1}^N\Big|\frac{u(x+he_i)-u(x)}{h}\Big|^2[1+(u(x+he_i)\cdot u(x))]=2|\nabla u(x)|^2\,,
\end{eqnarray*}
where we also took into account that $u(x+he_i)\cdot u(x) \to |u(x)|^2=1$ as $h\to 0$.
\end{proof}

\smallskip

%%%%%
%%We denote by $\MM_1(S^{N-1})$ the space of probability measures on the sphere $S^{N-1}$ with null barycenter. Given $Q\in\Mnn$, we denote by $|Q|$ the Frobenius norm of $Q$. Given a Young measure $\mu \colon \Om \mapsto \MM_1(S^{N-1})$ we define the function $Q_{\mu} \colon \Om \mapsto \Mnn$ as the matrix of second moments of $\mu_x$ at every $x \in \Om$, namely
%\begin{equation}\label{qu}
%Q_{\mu}(x):=\int_{S^{N-1}}{\xi \otimes \xi \,d\mu_x(\xi)}\,.
%\end{equation}
%It is a trivial fact that this mapping in linear in $\mu$.  
%\begin{proposition}
%Let $Q_\mu$ be defined as in \eqref{qu}. Then for a.e. $x \in \Om$ one has
%\begin{equation}\label{prop}
%Q_{\mu}\geq 0\quad\hbox{and}\quad{\rm tr}(Q_{\mu}(x))=1,
%\end{equation}
%Moreover $|Q_{\mu}(x)|\leq 1$ for every $x$.
%\end{proposition}
%\begin{proof}
%For every $v \in \Rn$ one has
%$$
%\langle Q_{\mu}(x)v, v\rangle= \int_{S^{N-1}}{(\xi\cdot v)^2\,d\mu_x(\xi)}\ge 0\,.
%$$
%We also have that 
%$$
%{\rm tr}(Q_{\mu}(x))=\int_{S^{N-1}}{|\xi|^2\,d\mu_x(\xi)}=\int_{S^{N-1}}{d\mu_x(\xi)}=1\,.
%$$
%Let now $0\leq\lambda_1\leq\lambda_2\leq\dots\leq\lambda_N$ be the ordered eigenvalues of $Q_\mu$. Using the trace condition $\sum_{i=1}^N\lambda_i=1$ and since $|Q_\mu|=\sqrt{\sum_{i=1}^N\lambda_i^2}$, we easily conclude that $|Q_\mu|\leq 1$. 
%\end{proof}

In Section \ref{concentration} we will consider the distributional Jacobian $Jw$ of a function $w\in W^{1,1}(\Omega;\R^2)\cap L^\infty(\Omega;\R^2)$. It is defined through its action on test functions $\phi \in C^{0,1}_c(\Omega)$, the space of Lipschitz continuous functions on $\Omega$ with compact support, as follows:
\begin{equation}\label{jwds}
\langle J w, \varphi\rangle= - \int_{\R^2} w_1 (w_2)_{x_2} \varphi_{x_1} - w_1
(w_2)_{x_1} \varphi_{x_2} \, dx.
\end{equation}
It is not difficult to see that $w\mapsto Jw$ is continuous as a map from $W^{1,1}(\Omega;\R^2)\cap L^{\infty}(\Omega;\R^2)$ to the dual of  $C^{0,1}_c(\Omega)$; moreover, if additionally $w\in W^{1,2}(\Omega;\R^2)$, then $Jw \in L^1(\Om)$ and %\eqref{jwds} reduces to 
we recover the usual definition of Jacobian as $Jw={\rm det }\nabla w$.

\section{The energy model: the bulk scaling}\label{section:abstract}

%Given $\e >0$ and $u\colon \e\Z^N \cap \Om \mapsto S^{N-1}$, we define $u_i:=u(\e i)$ and we consider the %energies...
%\begin{equation}\label{energy}
%E_\e(u):=\sum_{|i-j|=1}{\e^N f(u_i, u_j)}.
%\end{equation}

%Since we identify vectors with same direction but opposite orientation, about the energy $f$ we will assume that it is insensitive to orientation changes, together with symmetry; precisely we assume
%\begin{equation}\label{hipp}
%f(u,v)=f(v,u)=f(-u,v)
%\end{equation}
%for every $u$ and $v \in S^{N-1}$; otherwise stated, the energy $f$ is defined as a function on the projective space $\mathbb{P}^{N-1}$.

Given $\Omega\subset\R^N$ and $\e>0$, we consider a
pairwise-interacting discrete system on the lattice $\Z_\e(\Omega)$ whose state variable is denoted by  
$u:\e\Z_\e(\Omega)\to\R^N$. Such a system is driven by an energy $E_\e:\R^N\to(-\infty,+\infty)$ given by
\begin{eqnarray*}\label{32-eq:ener-old-var}
E_\e(u)=\sum_{\alpha,\beta \in \Z_\e(\Omega)} \e^N
e_\e(\e\alpha,\e\beta,u(\e\alpha),u(\e\beta))
\end{eqnarray*}
for some energy density
$e_\e:(\e\Z_\e(\Omega))^2\times\R^{2N}\to\R$.
We observe that there is no loss of generality in assuming the
interactions symmetric. This symmetry condition is expressed by the
formula $$e_\e (\e\alpha,\e\beta,u,v)=e_\e (\e\beta,\e\alpha,v,u)$$ (note
that, otherwise, one could deal with $\tilde{e}_\e
(\e\alpha,\e\beta,u,v)= \frac{1}{2}(e_\e (\e\beta,\e\alpha,v,u)+e_\e
(\e\alpha,\e\beta,u,v))$). 
A key feature of our model is its {\em orientational symmetry}; i.e. the systems we consider are characterized by the property that one cannot distinguish a state from its antipodal. From the point of view of the energy, this translates into the following condition:
\begin{equation}\label{def:symmetry}
e_\e (\e\alpha,\e\beta,u,v)=e_\e (\e\alpha,\e\beta,u,-v).
\end{equation}

In the following we find it useful to rewrite the energy by a change
of variable. Given $\xi\in\zN$ we define:
\begin{equation}
g_\e^\xi(\e\al,u,v):=e_\e(\e\al,\e\al+\e\xi,u,v)
\end{equation}
and we have
\begin{eqnarray*}
E_\e(u)=\sum_{\xi\in \zN }\sum_{\alpha\in R_\e^\xi(\Omega)} \e^N
g^\xi_\e(\e\alpha,u(\e\alpha),u(\e\al+\e\xi)),
\end{eqnarray*}
with $R_\e^\xi(\Omega):=\{\a\in\Z_\e(\Om):\ \a+\xi\in\Z_\e(\Om)\}$.
Note that, in the current variables, the symmetry conditions read
\begin{eqnarray}
g_\e^\xi (\e\alpha,u,v)&=&g_\e^{-\xi}(\e\alpha+\e\xi,v,u)\label{hsymmetry1}\\
g_\e^\xi (\e\alpha,u,v)&=&g_\e^{\xi}(\e\alpha,u,-v)\label{hsymmetry2}
\end{eqnarray}
Notice that the two above equations also imply that 
\begin{equation}\label{hsymmetry1bis}
g_\e^\xi (\e\alpha,-u,v)=g_\e^\xi (\e\alpha,u,v)=g_\e^\xi (\e\alpha,-u,-v).
\end{equation}
Indeed, by \eqref{hsymmetry1} and \eqref{hsymmetry2}
$$
g_\e^\xi (\e\alpha,-u,v)=g_\e^{-\xi} (\e\alpha+\e\xi,v,-u)=g_\e^{-\xi} (\e\alpha+\e\xi,v,u)=g_\e^{\xi} (\e\alpha,u,v)
$$
and the other equality can be proven similarly.
\subsection {Q-theory - $L^\infty$ energies}
In the rest of the paper we will be concerned with energies defined on $S^{N-1}$-valued functions $u:\e\Z_\e(\Omega)\to S^{N-1}$. In this case one can regard these energies as defined on tensor products of the type $Q(u)=u\otimes u$. More precisely, for all $\e,\ \alpha,\ u,\ v$ we will write 
\begin{equation}\label{Q_energy}
f_\e^\xi (\e\alpha,Q(u),Q(v)):=g_\e^\xi (\e\alpha,u,v).
\end{equation}
This identification is not ambiguous because of \eqref{hsymmetry1bis}, \eqref{hsymmetry2} and the following proposition.
\begin{proposition}
Let $u,w\in S^{N-1}$. Then $Q(u)=Q(w)$ if and only if $u=\pm w$.
\end{proposition}
\begin{proof}
By the definition of $Q$, \eqref{prodotti} and \eqref{modul} we have that
\begin{equation*}
|Q(u)-Q(w)|^2=|u\otimes u-w\otimes w|^2=2(1-(u\cdot w)^2).
\end{equation*}
Therefore $Q(u)=Q(w)$ if and only if $(u\cdot w)^2=1$. Since $u,w\in S^{N-1}$ this is equivalent to the statement.
\end{proof}

The choice of the variables in \eqref{Q_energy} will prove to be very useful in the following analysis and corresponds to the usual de Gennes $Q$-tensor approach to liquid crystals (see \cite{DeGPro}). In these variables the two symmetry conditions \eqref{hsymmetry1} and \eqref{hsymmetry2} reduce only to the first one \eqref{hsymmetry1} which now reads as
\begin{eqnarray}\label{fsymmetry}
f_\e^\xi (\e\alpha,Q(u),Q(v))=f_\e^{-\xi}(\e\alpha+\e\xi,Q(v),Q(u)),
\end{eqnarray}
the second symmetry condition being entailed by the structure of the tensor variable.
We define $S^{N-1}_\otimes\subset \Mnn$ as $S^{N-1}_\otimes:=\{Q(v):\ v\in S^{N-1}\}$.
Note that by \eqref{modul} we have $|Q|=1$ for all $Q\in S^{N-1}_\otimes$.
We then identify every $u\colon \e\Z_\e(\Om) \mapsto S^{N-1}$ with $Q(u)\colon \e\Z_\e(\Om)\mapsto S^{N-1}_\otimes$. Furthermore, to the latter we associate a piecewise-constant interpolation belonging to the class
\begin{equation}\label{cepsB}
C_\e(\Om; S^{N-1}_\otimes):=\{Q \colon \Om \to S^{N-1}_\otimes: Q(x)=Q(\e i)\,\,\hbox{ if } x \in \e\{i+W_N\},\,i \in \Z_\e(\Om)\}. 
\end{equation}
As a consequence we may see the family of energies $E_\e$ as defined on a subset
of $L^\infty(\Omega,\Mnn)$ and consider their extension on
$L^\infty(\Omega,\Mnn)$ through the family of functionals
$F_\e:L^\infty(\Omega,\Mnn)\to(-\infty,+\infty]$ defined as
\begin{eqnarray}\label{eq:energie-eps}
F_\e(Q)=\begin{cases}\displaystyle \sum\limits_{\xi\in \zN }
\sum\limits_{ \alpha\in R_\e^\xi(\Omega)}\e^N f_\e^\xi
(\e\alpha,Q(\e\alpha),Q(\e\alpha+\e\xi))
 & \text{if $Q\in C_\e(\Omega,S^{N-1}_\otimes)$}\cr
+\infty & \text{otherwise.}\cr\end{cases}
\end{eqnarray}

\medskip

We make the following set of hypotheses on the family of functions 
$f_\e^\xi:\e\Z_\e(\Om)\times\Mnn\times\Mnn\to(-\infty,+\infty]$:
\begin{itemize}
\item[(H1)]For all $\alpha$, $\xi$ and $\e$, $f_\e^\xi$ satisfies \eqref{fsymmetry},\\
\item[(H2)] For all $\alpha$, $\xi$ and $\e$, $f_\e^\xi(\e\alpha,Q_1,Q_2)=+\infty$ if $(Q_1,Q_2)\notin S^{N-1}_\otimes\times S^{N-1}_\otimes$,
\item[(H3)] For all $\alpha$, $\xi$ and $\e$, there exists $C_{\e,\alpha}^\xi \geq 0$ such that
\begin{eqnarray*}
&&|f^{\xi}_\e (\e\alpha,Q_1,Q_2)| \leq C_{\e,\alpha}^\xi\quad \hbox{ for
all } Q_1,Q_2\in S^{N-1}_\otimes,\cr\cr &&\dps{\limsup_{\e \to 0}
 \sup_{\alpha\in\Z_\e(\Omega)}\sum_{\xi\in \mathbb{Z}^N} C_{\e,\alpha}^\xi<\infty},
\end{eqnarray*}
\item[(H4)] for all $\delta >0$, there exists $M_\delta>0$ such that
\begin{eqnarray*}
\dps{\limsup_{\e\to 0} \sup_{\alpha\in\Z_\e(\Omega)}\sum_{|\xi|\geq
M_\delta} C_{\e,\alpha}^\xi\leq \delta  }.
\end{eqnarray*}
\end{itemize}

In what follows we will also use a localized version of the functional $F_\e$, defined below. Let ${\mathcal A}(\Om)$ be the class of all open subset of $\Omega$. For every $A\in{\mathcal A}(\Om)$ we set

\begin{eqnarray}\label{eq:energie-eps-loc}
F_\e(Q,A)=\begin{cases}\displaystyle\sum\limits_{\xi\in \zN }\sum\limits_{ \alpha
\in R_\e^\xi(A)}\e^N f_\e^\xi
(\e\alpha,Q(\e\alpha),Q(\e\alpha+\xi))
 & \text{if $Q\in C_\e(\Omega;S^{N-1}_\otimes)$}\cr
+\infty & \text{otherwise.}\cr\end{cases}
\end{eqnarray}

%%%%%%%%%%%%%%%%%%%%%%%%%%%%%%%%%%%%%%%%%%%%%%%%%%%%%%

%%%%%%%%%%%%%%%%%%%%%%%%%%%%%%%%%%%%%%%%%%%%%%%%%%%%

\subsection{General integral representation theorems}

In this section we state a general compactness and integral representation result for the functionals $F_\e$ defined in \eqref{eq:energie-eps-loc}. To that end we will need the following characterization of the convex envelope of $S^{N-1}_\otimes$.

\begin{proposition}\label{convex envelope} The convex envelope of $S^{N-1}_\otimes$ is given by the set
\begin{equation}\label{convex}
K:=\{Q \in \Mnn: Q\geq 0,\ {\rm tr}\,Q=1\}.
\end{equation}
\end{proposition}
\begin{proof}
Note that $K$ is convex and that it contains  the convex envelope of $(S^{N-1}_\otimes)$. It therefore remains to show that every matrix in $K$ can be represented as a convex combination of matrices in $S_\otimes^{N-1}$. For every $Q\in K$, by the symmetry and the positive semidefiniteness of $Q$, we may consider its ordered  eigenvalues $0\leq\lambda_1\leq\lambda_2\leq\dots\leq\lambda_N$. By the trace condition we have that 
\begin{equation}\label{convex envelope: trace}
\sum_{i=1}^N\lambda_i=1.
\end{equation}
On the other hand we may represent $Q$ as
\begin{equation}\label{convex envelope: representation}
Q=\sum_{i=1}^N\lambda_ie_i\otimes e_i,
\end{equation}
where $\{e_1,e_2,\dots,e_N\}$ is an orthonormal basis in $\Rn$ and each of the $e_i$ is an eigenvector relative to $\lambda_i$. Since $\lambda_i\geq 0$, combining \eqref{convex envelope: trace} and \eqref{convex envelope: representation} we conclude the proof.
\end{proof}

The following
$\Gamma$-convergence result holds true.

\begin{theorem}[compactness and integral representation]\label{32-th:disc-cont-Li-mean}
Let $\{ f_\e^{\xi}\}$ satisfy hypotheses {\rm (H1)--(H4)}, and let $K$ be given by {\rm(\ref{convex})}. Then, for every sequence $\e_j$ converging to zero, there exists a subsequence (not relabelled) and a Carath\'eodory function $f:\Omega \times
K\to\R$ convex in the second variable such that
the functionals
$F_{\e_{j}}$ $\Gamma$-converge with respect to
the weak$^*$-convergence of $L^{\infty}(\Omega,\Mnn)$ to the
functional $F:L^{\infty}(\Omega,\Mnn)\to(-\infty,+\infty]$ given by
\begin{eqnarray*}
F(Q)=\begin{cases}\displaystyle \int_\Omega f(x,Q(x))\ dx&\text{if $Q\in
L^{\infty}(\Omega;K)$},\cr +\infty& otherwise.\end{cases}
\end{eqnarray*}
\end{theorem}

\begin{proof}
The proof follows from \cite[Theorem 3.4]{ACG} upon identifying $\Mnn$ with $\R^{N(N+1)/2}$ via the usual isomorphism, and by the characterization of the convex envelope of $S^{N-1}_\otimes$ in Proposition \ref{convex envelope}.
\end{proof}
\medskip

In what follows we state a homogenization problem for our type of energies.

Let $k\in \mathbb{N}$ and for any $\xi \in \mathbb{Z}^N$, let
$f^\xi:\mathbb{Z}^N\times S^{N-1}_\otimes \times S^{N-1}_\otimes\to
\mathbb{R}$ be such that $f^\xi(\cdot,Q_1,Q_2)$ is $[0,k]^N$-periodic
for any $Q_1,Q_2\in  S^{N-1}_\otimes$. We then set
\begin{equation}\label{32-eq:hypo-perio}
f^\xi_\e(\e\alpha,Q_1,Q_2):=f^\xi(\alpha,Q_1,Q_2).
\end{equation}
In this case, hypotheses (H2), (H3), (H4) read:
\begin{itemize}
\item[(H2')] For all $\alpha$ and $\xi$,
$\dps{f^\xi(\alpha,Q_1,Q_2)=+\infty}$ if $(Q_1,Q_2)\not\in S^{N-1}_\otimes\times S^{N-1}_\otimes$.
\item[(H3')] For all $\alpha$ and $\xi$, there exists $C^\xi\geq 0$ such that
$\dps{|f^\xi(\alpha,Q_1,Q_2)|\leq C^\xi}$ for all $Q_1,Q_2\in S^{N-1}_\otimes$, and
$\sum_\xi C^\xi <\infty$.
\end{itemize}
We now introduce the notion of discrete average.
\begin{definition}
For any $A\subset\Omega$, $\e>0$, and $Q\in C_\e(\Omega,S^{N-1}_\otimes)$, we
set
$$\langle Q \rangle^{d,\e}_A=\frac{1}{\#(\e \zN\cap A)}\dps{\sum_{\alpha  \in \e \mathbb{Z}^N\cap A}Q(\alpha)}$$
(in this notation $d$ stands for {\em discrete}).
\end{definition}
%Given $\bar Q\in co(S^{N-1}_\otimes)$, for all $\rho>0$ we define $F_\e^{\bar Q,\rho}:L^\infty(\Omega;\Mnn)\times{\mathcal
%A}(\Omega)\to(-\infty,+\infty]$ by
%\begin{eqnarray}\label{32-eq:energie-eps-mean}
%F_\e^{\bar Q,\rho}(Q,A)=\begin{cases}F_\e(Q,A)&\text{$\langle
%Q\rangle_A^{\e,d}\in {\overline B(\bar Q,\rho)}$}\cr +\infty& otherwise.\end{cases}
%\end{eqnarray}

\begin{theorem}[homogenization]\label{32-th:homog-Li}
Let $\{f_\e^\xi \}_{\e,\xi}$ satisfy \eqref{32-eq:hypo-perio}, {\rm (H1)}, {\rm(H2')}
and {\rm(H3')}. Then $F_\e$ $\Gamma$-converge  with respect to the $L^\infty$-weak$^*$ topology to the functional $F_{\rm hom}:L^{\infty}(\Omega;\Mnn)\to[0,+\infty]$ defined as
\begin{equation}\label{Fhom}
F_{\rm hom}(Q)=
\begin{cases}\displaystyle
\int_{\Omega} f_{\rm hom}(Q(x))dx &\text{if } Q\in L^{\infty}(\Omega;K)\\
+\infty&\text{otherwise,}
\end{cases}
\end{equation}
where $f_{\rm hom}$ is given by the {\em homogenization formula}
\begin{equation}\label{32-eq:homog-form_li}
f_{\rm hom}(Q)=\lim_{\rho\to 0}\lim_{h\to + \infty} \frac{1}{h^N} \inf
\Biggl\{ \sum_{\xi \in \zN} \sum_{\beta \in R^\xi_1(hW_N)}
f^\xi(\beta,Q(\beta),Q(\beta+\xi)),\langle Q \rangle^{d,1}_{hW_N}\in
{\overline B(\bar Q,\rho)} \Biggr\}.
\end{equation}
\end{theorem}

\begin{proof}The proof follows by applying \cite[Theorem 5.3]{ACG}.\end{proof}

%%%%%%%%%%%%%%%%%%%%%%%%%%%%%%%%%%%%%%%%%%%%%%%%%%%%%

\begin{remark}
The heuristics of the model we propose is the following. To each point $\e i$ belonging to the microscopic lattice $\e\Z_\e$, we associate an orientation $u^i$ and define an energy accounting for (long range) interactions among orientations. This results in a $\Gamma$-limit energy obtained relaxing the microscopic energy over all the possible mesoscales  $\delta_\e$, with $\e<<\delta_\e<<1$. 

In other models (see \cite{DeGPro}) a different point of view is taken. There, to each point $\e i\in \e\Z_\e(\Om)$ one associate a probability measure $\mu_i$ on $S^{N-1}$ accounting for an heuristic averaging of the microscopic orientations on some fixed mesoscale. Assuming symmetry properties of the distribution of these orientations one has that $\mu_i$ has vanishing barycenter. As a result a meaningful energy to consider depends on the second moment $Q_\mu$ (Q-tensor) associated to $\mu$. It may be seen that $Q_\mu\in K$ and that every $Q\in K$ is the second moment of some probability measure with vanishing barycenter. On the other hand there is no canonical way to associate with continuity a measure to a $Q$ tensor and the natural compactness of such energies only concerns the $Q$-tensors. Indeed even Sobolev compactness of $Q_{\mu_\e}$ is still weaker than any weak convergence of $\mu_\e$. As a result it seems unnatural to perform the $\Gamma$-limit with respect to a weak convergence of 
measures.  
\end{remark}
\subsection{Nearest-neighbor interactions in two dimensions}\label{2d}
In this section we consider a pairwise energy among nearest-neighboring points on a planar lattice whose configuration is parameterized by a function $u$ belonging to $S^1$. In this case the energy takes the form 
\begin{equation}\label{energy}
E_\e(u)=\sum_{|i-j|=1}\e^2f(u_i,u_j).
\end{equation}
The symmetry hypotheses can be rewritten as
\begin{equation}
f(u,v)=f(u,-v)=f(-u,v),
\end{equation}
and hypothesis (H3) reduces to assume that $f:\Z_\e(\Om)\to\R$ is a bounded function. 
Via the usual identification $Q(\e i)=Q(u_i)$ we may associate to $E_\e(u)$ the functional $F_\e(Q)$ given by:
\begin{eqnarray}\label{energy0thorder}
F_\e(Q)=
\begin{cases}\displaystyle
\sum_{|i-j|=1}\e^2f(u_i,\ u_j)& \text{if $ Q\in C_\e(\Om;K)$}\\
+\infty&\text{otherwise}.
\end{cases}
\end{eqnarray}
In this special case it is particularly easy to recover an explicit formula for the energy density $f_{\rm hom}$ in Theorem \ref{32-th:homog-Li}. Some useful identities are pointed out in the next proposition.
\begin{proposition}\label{K_prop}It holds that:
\begin{itemize}
\item[(a)] Let $u$ and $v \in S^{1}$. Denote by $I$ the identity $2\times 2$ matrix. Then
\begin{equation}\label{identity}
\Big|\frac12(u\otimes u+v \otimes v)-\frac12 I\Big|=\frac{\sqrt2}{2} | u \cdot v | \,.
\end{equation}
\item[(b)] Let $K$ be defined as in \eqref{convex} with $N=2$. Then we have
\begin{equation}\label{KN2}
K=\{Q\in \Mdue: |Q|\leq 1,\ {\rm tr}\,Q=1\}=\{Q \in \Mdue: \quad |Q-\tfrac12 I|\le\tfrac{\sqrt2}{2}\,,{\rm tr}\,Q=1 \}\,.
\end{equation}
\end{itemize}
\end{proposition}

%Using \eqref{TRt} and \eqref{KN2} we easily have that
%\begin{equation}\label{Kappa}
%K=
%\end{equation}

\begin{proof}
Let $Q\in\Mdue$ be such that ${\rm tr}\,Q=1$, then 
\begin{equation}\label{moduloQ}
\Bigl|Q-\frac{1}{2}I\Bigr|^2=|Q|^2-\frac{1}{2}.
\end{equation}
In the particular case where $Q=\frac{u\otimes u+v\otimes v}{2}$ with $u,\ v\in S^1$, using \eqref{prodotti} and \eqref{modul} we get
\begin{eqnarray*}
\left|\frac{u\otimes u+v\otimes v}{2}-\frac{1}{2}I\right|^2&=&\left|\frac{u\otimes u+v\otimes v}{2}\right|^2 -\frac{1}{2}=\frac{1}{4}(2+2(u\cdot v)^2)-\frac{1}{2}=\frac12(u\cdot v)^2
\end{eqnarray*}
which implies $(a)$.
Let $Q\in\Mdue$ be such that ${\tr \,Q}=1$. If $\lambda$ and $1-\lambda$ are the eigenvalues of $Q$, 
then the condition $1\ge|Q|^2=\lambda^2+(1-\lambda)^2$ is equivalent to $\lambda \in [0,1]$.
This implies the first equality in \eqref{KN2}. From that and \eqref{moduloQ} also the second equality follows.  \end{proof}

\begin{remark} 
Note that, following the same argument of the proposition above, we may also show that
$$
K=\{Q\in \Mdue:  | Q- s I |^2\le s^2+(1-s)^2, \ {\rm tr }\,Q=1\}.
$$
The case $s=\frac12$ highlighted in the proposition will be more useful since in that case  $Q-\frac12 I$
are traceless matrices.
\end{remark}

We now show that, for every $Q \in K$, there exist $u$ and $v \in S^1$ such that
\begin{equation*}
Q= \tfrac12 (u\otimes u+ v \otimes v)\,.
\end{equation*}
This decomposition will turn out to be unique (up to the order) except in the case when $Q=\frac12 I$. Indeed any pair  of orthonormal vectors $u,\ u^\perp$ satisfies
\begin{equation*}
\frac{u\otimes u+u^\perp\otimes u^\perp}{2}=\frac12 I.
\end{equation*}

\begin{proposition}\label{formula_magica}
Let $Q \in K$. Then there exist $u$ and $v \in S^1$ such that
\begin{equation}\label{Kap2}
Q= \tfrac12 (u\otimes u+ v \otimes v)\,.
\end{equation}
Furthermore, if $Q \neq \tfrac 12 I$, the matrices $u\otimes u$ and $v \otimes v$ are uniquely determined up to exchanging one with the other.
\end{proposition}

\begin{proof}
Let $Q\in K$. Then it exists $\lambda \in [0,1]$ and an orthonormal basis $\{n_1, n_2\}$ of $\R^2$ such that
$$
Q=\lambda n_1\otimes n_1 +(1-\lambda) n_2\otimes n_2\,.
$$
If we set
\begin{equation*}
u=\sqrt{\lambda}\, n_1+\sqrt{1-\lambda}\,n_2\,,\quad v=\sqrt{\lambda}\,n_1-\sqrt{1-\lambda}\,n_2\,,
\end{equation*}
both $u$ and $v\in S^1$, and a direct computation shows that \eqref{Kap2} is satisfied.

Now, let us suppose that $Q=\frac12 (u\otimes u+ v \otimes v)=\frac12 (z\otimes z+ w \otimes w)$, with $u$, $v$, $z$, $w\in S^1$, and $Q\neq \frac 12 I$. First of all, by \eqref{identity} we get
$$
|u\cdot v|=|z\cdot w|\neq 0\,.
$$
Up to changing $w$ with $-w$, which does not affect the matrix $w\otimes w$, we can indeed suppose
\begin{equation}\label{poi}
(u\cdot v)=(z\cdot w)\neq 0\,.
\end{equation}
Set $\lambda= \frac 12 (1+ (u\cdot v))$. Using \eqref{azione}, a direct computation and \eqref{poi} then give
$$
Q(u+v)=\lambda (u+v)\,,\quad Q(z+w)=\lambda (z+w)
$$
and
$$
Q(u-v)=(1-\lambda) (u-v)\,,\quad Q(z-w)=(1-\lambda) (z-w)
$$
Since, by \eqref{poi}, $\lambda \neq \frac12$, the matrix $Q$ has two distinct one-dimensional eigenspaces. Hence the vector $u+v$ must then be parallel to $z+w$ and $u-v$ parallel to $z-w$. Furthermore, again using \eqref{poi}, $|u+v|=|z+w|$ and $|u-v|=|z-w|$. Up to changing both $u$ and $v$ with their antipodal vectors $-u$ and $-v$ (which does not affect the matrices $u\otimes u$ and $v\otimes v$) we can indeed suppose 
\begin{equation}\label{ugua1}
u+v=z+w
\end{equation}
while leaving \eqref{poi} unchanged. 
Then, if necessary exchanging $z$ and $w$ we can additionally assume 
\begin{equation}\label{ugua2}
u-v=z-w\,;
\end{equation}
again, this would not affect the validity of \eqref{poi} and \eqref{ugua1}. Then, \eqref{ugua1} and \eqref{ugua2} are simultaneously satisfied if and only if $u=z$ and $v=w$, as required.
\end{proof}

In order to compute the $\Gamma$-limit of $E_\e(u)$ given by \eqref{energy} we will use a dual-lattice approach. To that end we need to fix some notation about lattices and the corresponding interpolations. 
Given 
$
u\colon \e\Z^2 \cap \Om \to S^1
$
we set as usual $Q_i=Q(u_i)=u(\e i)\otimes u(\e i)$ for every $i\in \Z_\e(\Om)$ and we identify $u$ with a piecewise-constant interpolation $Q_\e$ by defining the class
\begin{equation}\label{ceps1}
C_\e(\Om; K):=\{Q \colon \Om \to K: Q(x)=Q^i\,\hbox{ if } x \in \e\{i+W\},\,i \in \Z_\e(\Om)\} .
\end{equation}
We also associate to $u$ a piecewise-constant interpolation on the 'dual' lattice 
\begin{equation}\label{dual}
\Z_\e'(\Om):=\Bigl\{\frac{i+j}{2}: i, j \in \Z^2,\, |i-j|=1,\, \e i, \e j \in \Om\Bigr\}\,,
\end{equation}
by setting
\begin{equation}\label{rel}
Q^k=\frac12 (Q_i+ Q_j)
\end{equation}
for $k=\frac{i+j}{2}$. Correspondingly we define $Q'_\e$ in the class
\begin{eqnarray}\label{ceps}\nonumber
C'_\e(\Om; K):=\Bigl\{Q \colon \Om \to K: \{Q_i\}\hbox{ exist such that }
Q(x)=\frac12 (Q_i+ Q_j)\\
\,\hbox{ if }x \in \e\Bigl\{\frac{i+j}{2}+W'\Bigr\},i, j \in \Z^2,\, |i-j|=1\Bigr\}\,, 
\end{eqnarray}
where $W'$ is the reference cube of the dual configuration, obtained from the cube $\bigl[-\frac{\sqrt{2}}4, \frac{\sqrt{2}}4\bigr)^2$ by a rotation of $\frac \pi 4$.
The following lemma asserts that these two interpolations have the same limit points as $\e$ goes to $0$. 

\begin{lemma}\label{easy}
Let a family of functions $u_\e\colon \e\Z_\e(\Om) \to S^1$ be given and define accordingly the piecewise-constant interpolations $Q_\e \in C_\e(\Om;K)$ and $Q'_\e \in C'_\e(\Om;K)$, respectively. Then
\begin{equation*}
Q_\e-Q'_\e \wtos 0
\end{equation*}
weakly$^*$ in $L^\infty(\Om;K)$ as $\e\to 0$.
\end{lemma}
\begin{proof}
Let $\Om' \subset \subset \Om$, and $i \in \Z_\e(\Om')$ such that ${\rm dist}(\e i,\partial \Om')\ge \e$. This implies that $\e\{i+W\}\subset \Om'$, and that for every of $k$ satisfying $|k-i|=\frac12$, one has $k \in \Z_\e'(\Om)$ and $\e\{k+W'\}\subset \Om'$. 
Now, in dimension $N=2$ there are $4$ of such $k$'s. Since $W$ has unit area and $W'$ has area $\frac12$, one therefore gets
$$
\int_{\e\{i+W\}}A:Q_i \,dx=(A:Q_i)=A:\frac{Q_i}{2}\sum_{\substack{k \in \Z_\e'(\Om)\\ |k-i|=\frac12}}\e^2|\{k+W'\}|=\sum_{\substack{k \in \Z_\e'(\Om)\\ |k-i|=\frac12}}\int_{\e\{k+W'\}}A:\frac{Q_i}{2} \,dx
$$
for every $A \in \Mdue$.
Summing over all indices $i$, it follows that
$$
\int_{\Om'}A:(Q_\e(x)-Q'_\e(x))\,dx \to 0
$$
for every $A \in \Mdue$ and $\Om' \subset \subset \Om$, which implies the conclusion by the uniform boundedness of $Q_\e$ and $Q'_\e$.
\end{proof}

We now set
\begin{equation}\label{fhat}
\widehat f(Q):=
\begin{cases}
f(u,v)&\text{if }\frac{u\otimes u+v\otimes v}{2}=Q\neq\frac{1}{2}I\\
\min\{f(u,v):\ u,\ v\in S^1,\ \frac{u\otimes u+v\otimes v}{2}=\frac12 I\} &\text{if }Q=\frac12 I
\end{cases}
\end{equation}
and notice that
%It follows now easily from \eqref{energy} and \eqref{summa} that, if $u\colon \e \Z^2 \cap \Om \to S^1$ and $Q'\in C'_\e(\Om; K)$ are related by \eqref{rel} and \eqref{ceps}, one has
\begin{equation}\label{trivialestimate}
E_\e(u)\geq 2 \sum_{k \in \Z_\e'(\Om)} \e^2 \widehat f(Q^k).
\end{equation}
We also define on $L^\infty(\Om; \Mdue)$ the energy $\widehat F_\e$ by setting
\begin{equation}\label{energy2}
\widehat F_\e(Q):=\begin{cases}
         \displaystyle
         2\sum_{k \in \Z_\e'(\Om)} \e^2 \widehat f(Q^k)\quad\hbox{if }Q'\in C'_\e(\Om; K)\\
         +\infty \quad\hbox{otherwise}
         \end{cases}
\end{equation}
and we observe that 
\begin{equation}\label{energy2+}
\widehat F_\e(Q'):=\begin{cases}
         \displaystyle
         4\int_{\Om} \widehat f(Q'(x))\,dx+ r_\e\quad\hbox{if }Q'\in C'_\e(\Om; K)\\
         +\infty \quad\hbox{otherwise,}
         \end{cases}
\end{equation}
where the reminder term $r_\e$ comes from the fact that a portion of the cubes $k+\e W'$ may not be completely contained in $\Om$, and a multiplier $2$ appears, since the area of the reference cube $W'$ is $\frac12$. 
By the boundedness of the integrand we have that $r_\e=o(1)$ uniformly in $Q'\in C'_\e(\Om; K)$. 
We then have the following result.

\begin{theorem}\label{theorem:nn_general}
Let $F_\e\colon L^\infty(\Om;\Mdue)\to \R \cup\{\infty\}$ be defined by \eqref{energy0thorder}. Then $F_\e$ $\Gamma$-converges with respect to the weak$^*$-topology of $L^\infty(\Om;\Mdue)$ to the functional $F\colon L^\infty(\Om;\Mdue)\to \R \cup\{\infty\}$ defined by
$$
F(Q):=\begin{cases}
         \displaystyle
         4\int_{\Om} \widehat f^{**}(Q(x))\,dx\quad\hbox{if }Q \in L^\infty(\Om; K)\\
         +\infty \quad\hbox{otherwise}.
         \end{cases}
$$
In particular $f_{\rm hom}=4\widehat f^{**}$.
\end{theorem}
\begin{proof}The proof is obtained by using \eqref{trivialestimate}, 
\eqref{fhat} and Lemma \ref{easy}.\end{proof}

\medskip

Being of particular interest in the applications, we now further simplify the formula above in the isotropic case, that is when
\begin{equation*}
f(Ru,Rv)=f(u,v),\quad \hbox{ if } u,v\in S^1\ \hbox{and } R\in SO(2).
\end{equation*}
The previous condition is equivalent to saying that $f$ is a function of the scalar product $u \cdot v$; taking also into account the head-to-tail symmetry condition \eqref{def:symmetry} this amounts to require that there exists a bounded Borel function
$
h\colon [0,1]\to \R
$
such that
\begin{equation}\label{sfer}
f(u,v)=h( |u\cdot v|)
\end{equation}
for every $u$ and $v \in S^{1}$. This choice of the energy density, together with the special symmetry of $K$ in the two dimensional case (highlighted in formula \eqref{KN2}) will lead to a radially symmetric $f_{\rm hom}$. 

We first observe that from \eqref{sfer} and \eqref{identity}, the function $\widehat f$ defined in \eqref{fhat} takes now the form
\begin{equation}\label{grt}
\widehat f(Q)= h(\sqrt2 |Q-\tfrac12 I|)
\end{equation}
for every $Q\in K$. For such a $\widehat f$ the function $\widehat f^{**}$ can be characterized by a result in convex analysis. To that end, we introduce a monotone envelope as follows

\begin{definition}\label{h+++}
Let $h:[0,1]\to\R$. We define the {\em nondecreasing convex lower-semicontinuous envelope} $h^{++}$ of $h$ as the largest nondecreasing convex lower-semicontinuous function below $h$ at every point in $[0,1]$. This is a good definition since all these properties are stable when we take the supremum of a family of functions.
\end{definition}

\begin{remark}\label{r: h++}
(i) The function $h^{++}$ satisfies the following property
\begin{equation}\label{infimum}
\min_{t \in [0,1]}h^{++}(t)=h^{++}(0)=\inf_{t \in [0,1]}h(t)\,.
\end{equation}
As a consequence, if $h$ has an interior minimum point $\ol t$, then
\begin{equation}\label{infimum2}
h^{++}(t)\equiv h(\ol t)
\end{equation}
for every $0\le t\le \ol t$;

(ii) let $\Md$ be the subspace of trace-free symmetric $2\times 2$ matrices, and let $\varphi:\Md\to\R\cup\{+\infty\}$ and $\psi:[0,+\infty)\to \R\cup\{+\infty\}$ be proper Borel functions. If $\varphi(Q)=\psi(|Q|)$ for all $Q\in\Md$, then $\varphi^{**}(Q)$ coincides with the largest nondecreasing convex lower-semicontinuous function below $\psi$ at every point in $[0,+\infty)$ computed at $|Q|$ (\cite[Corollary 12.3.1 and Example below]{Roc}). Note that if $\psi(t)=+\infty$ for $t\ge 1$ then we have $\varphi^{**}(Q)=\psi^{++}(|Q|)$ if $|Q|\le 1$, with $\psi^{++}$ as in Definition \ref{h+++}.

\end{remark}

\begin{proposition}
Let $h:[0,1]\to \R$ be a bounded Borel function. 
%Define the monotone conjugate $h^+: \R \to \R$ of $h$ by setting
%$$
%h^+(\zeta)=\sup\{t\zeta-h(t): t\in [0,1]\}
%$$ 
%and $h^{++}:[0,1]\to \R$ by
%\begin{equation}\label{h++}
%h^{++}(t)=\sup\{\zeta t-h^+(\zeta): \zeta \ge 0\}\,. 
%\end{equation}
Let $K$ be given by \eqref{convex} and define $\widehat f\colon \Mdue\to \R$ by
$$
\widehat f(Q):= \begin{cases}
        h(\sqrt2 |Q-\tfrac12 I|)\quad\hbox{if }Q \in K\\
        +\infty \quad\hbox{otherwise}\,.
       \end{cases}
$$
Then the lower semicontinuous and convex envelope $\widehat f^{**}$ of $\widehat f$ is given by
\begin{equation}\label{g++}
\widehat f^{**}(Q)= \begin{cases}
        h^{++}(\sqrt2 |Q-\tfrac12 I|)\quad\hbox{if }Q \in K\\
        +\infty \quad\hbox{otherwise}\,.
       \end{cases}
\end{equation}
\end{proposition}

\begin{proof}
For a matrix $Q \in \Mdue$, consider the deviator $Q_D$ of $Q$, that is its projection onto the linear subspace $\Md$ of trace-free symmetric $2\times 2$ matrices, which is orthogonal to the identity. Using \eqref{KN2} it is not difficult to see that $Q\in K$ if and only if $Q=\frac 12 I + Q_D$ with $Q_D$ belonging to
$$
K_D:=\{ Q_D \in \Md: |Q_D|\le \tfrac{\sqrt2}{2}\}\,.
$$
Define now $\widehat f_D \colon \Md \to [0, +\infty]$ by
$$
\widehat f_D(Q_D):= \begin{cases}
        h(\sqrt2 |Q_D|)\quad\hbox{if }Q_D \in K_D\\
        +\infty \quad\hbox{otherwise}\,.
       \end{cases}
$$
Obviously, $\widehat f^{**}(Q)=+\infty$ when $Q\notin K$. When $Q \in K$, being $|Q-\frac 12 I|=|Q_D|$ and exploiting the well-known characterization
\begin{equation}\label{convexification}
\widehat f^{**}(Q)=\inf\Bigl\{\sum_{i=1}^m \lambda_i \widehat f(Q_i):\quad m\in \N,\,\lambda_i\ge 0\,,\sum_{i=1}^m \lambda_i=1,\,\sum_{i=1}^m \lambda_iQ_i=Q\Bigr\} 
\end{equation}
we easily get $\widehat f^{**}(Q)=\widehat f_D^{**}(Q_D)$, so that \eqref{g++} follows now by Remark \ref{r: h++}(ii).
\end{proof}

\begin{theorem}\label{theorem:nn}
Let $F_\e\colon L^\infty(\Om;\Mdue)\to \R \cup\{\infty\}$ be defined by \eqref{energy0thorder} with $f$ as in \eqref{sfer}. Then $F_\e$ $\Gamma$-converges with respect to the weak$^*$-topology of $L^\infty(\Om;\Mdue)$ to the functional $F\colon L^\infty(\Om;\Mdue)\to \R \cup\{\infty\}$ defined by
$$
F(Q):=\begin{cases}
         \displaystyle
         4\int_{\Om} h^{++}(\sqrt2 |Q(x)-\tfrac12 I|)\,dx\quad\hbox{if }Q \in L^\infty(\Om; K)\\
         +\infty \quad\hbox{otherwise}
         \end{cases}
$$
where $h^{++} \colon[0,1]\to \R$ is the nondecreasing convex lower semicontinuous envelope of $h$  as in Definition {\rm\ref{h+++}}.
\end{theorem}

\begin{proof} The result follows Theorem \ref{theorem:nn_general} and \eqref{g++}.\end{proof}

\smallskip
We end up this section with another simple example where the locus of minima of $f_{\rm hom}$ can be explicitly computed.

\begin{example}\label{noradial}\rm We consider $f$ of the form $f(u,v)=\min\{\tilde f(u,v),\tilde f(v,u)\}$ with
$$
\tilde f(u,v):=\begin{cases}
         0 &\hbox{if } |u\cdot e_1|=l,\,|u\cdot v|=m\\
        1 &\hbox{ otherwise},\,
        \end{cases}
$$
where for simplicity the parameters $l$ and $m$ are both taken strictly contained in $(0,1)$ (the case where at least one of them attains the value $0$ or the value $1$ can be treated with minor modifications).
Defining $\widehat f$ as in \eqref{fhat}, one immediately has that $\widehat f$ only takes the two values $0$ and $1$: namely, setting
$$
G:=\Bigl\{Q \in K: Q=\frac{u\otimes u + v\otimes v}{2},\quad |u\cdot e_1|=l,\,|u\cdot v|=m\Bigr\}
$$
one has $\widehat f(Q)=0$ if $Q\in K$ and $\widehat f(Q)=1$ if $Q \in K \setminus G$. In our case, since $0<l<1$ and $0<m<1$, $G$ consists of exactly $4$ distinct matrices. Precisely, taking $\theta_l$ and $\theta_m \in (0, \frac \pi 2)$ such that $\cos (\theta_l)=l$ and $\cos (\theta_m)= m$
we set
$$
u_1=(\cos (\theta_l), \sin (\theta_l)), \quad u_2=(\cos (\pi-\theta_l), \sin (\pi-\theta_l))
$$
and
$$
\begin{array}{c}
 v^1_1= (\cos (\theta_l+\theta_m), \sin (\theta_l+\theta_m)),\quad v^1_2= (\cos (\theta_l-\theta_m),\sin (\theta_l-\theta_m))\\[4pt]
 v^2_1= (\cos (\pi-\theta_l+\theta_m), \sin (\pi-\theta_l+\theta_m)),\quad v^2_2= (\cos (\pi-\theta_l-\theta_m), \sin (\pi-\theta_l-\theta_m))\,.
\end{array}
$$
Then, $G$ consists of the $4$ matrices $Q_i$, with $i=1,\dots,4$, respectively given by
$$
\begin{array}{c}
Q_1=\frac{u_1\otimes u_1 + v^1_1\otimes v^1_1}{2}=\frac 12\begin{pmatrix}
                                                             \cos^2(\theta_l)+\cos^2(\theta_l+\theta_m)&\frac12[\sin(2\theta_l)+\sin(2(\theta_l+\theta_m)]\\
                                                             \frac12[\sin(2\theta_l)+\sin(2(\theta_l+\theta_m))]&\sin^2(\theta_l)+\sin^2(\theta_l+\theta_m)
                                                             \end{pmatrix},\\[10pt]
Q_2=\frac{u_1\otimes u_1 + v^1_2\otimes v^1_2}{2}=\frac 12\begin{pmatrix}
                                                             \cos^2(\theta_l)+\cos^2(\theta_l-\theta_m)&\frac12[\sin(2\theta_l)+\sin(2(\theta_l-\theta_m))]\\
                                                             \frac12[\sin(2\theta_l)+\sin(2(\theta_l-\theta_m))]&\sin^2(\theta_l)+\sin^2(\theta_l-\theta_m)
                                                             \end{pmatrix},\\[10pt]
Q_3=\frac{u_2\otimes u_1 + v^2_1\otimes v^2_1}{2}=\frac 12\begin{pmatrix}
                                                             \cos^2(\theta_l)+\cos^2(\theta_l-\theta_m)&-\frac12[\sin(2\theta_l)+\sin(2(\theta_l-\theta_m))]\\
                                                             -\frac12[\sin(2\theta_l)+\sin(2(\theta_l-\theta_m))]&\sin^2(\theta_l)+\sin^2(\theta_l-\theta_m)
                                                             \end{pmatrix},\\[10pt]
Q_4=\frac{u_2\otimes u_2 + v^2_2\otimes v^2_2}{2}=\frac 12\begin{pmatrix}
                                                             \cos^2(\theta_l)+\cos^2(\theta_l+\theta_m)&-\frac12[\sin(2\theta_l)+\sin(2(\theta_l+\theta_m))]\\
                                                             -\frac12[\sin(2\theta_l)+\sin(2(\theta_l+\theta_m))]&\sin^2(\theta_l)+\sin^2(\theta_l+\theta_m)
                                                             \end{pmatrix}.

\end{array}
$$
By construction, and using \eqref{identity}, $|Q_i-\frac 12 I|=\frac{\sqrt{2}}{2}m$. Furthermore, $Q_1-Q_4$ is parallel to $Q_2-Q_3$. Therefore, in the two-dimensional affine space of the matrices $Q \in \Mdue$ with ${\rm tr}\,Q=1$, the convex envelope ${\rm co}(G)$ can be represented as a trapezoid inscribed in the circle with center at $\frac12 I$ and radius $\frac{\sqrt{2}}{2}m$ (see Figure \ref{fig: co_K_}).
Finally, by Theorem \ref{theorem:nn_general}, we have $f_{\rm hom}=4\widehat f^{**}$, which implies in this case that $f_{\rm hom}\ge 0$ and $f_{\rm hom}(Q)=0$ if and only if $Q\in{\rm co}(G)$. 
\end{example}
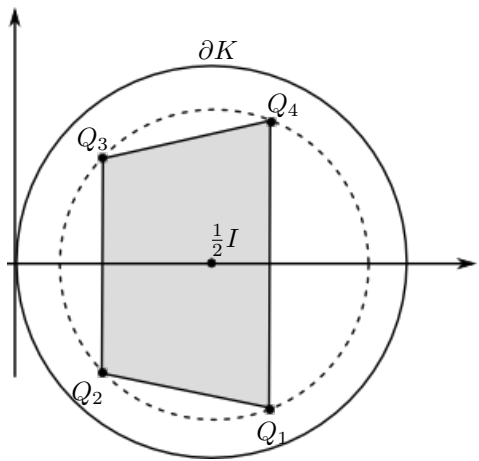
\begin{figure}
\begin{center}
\includegraphics[scale=.75 ]{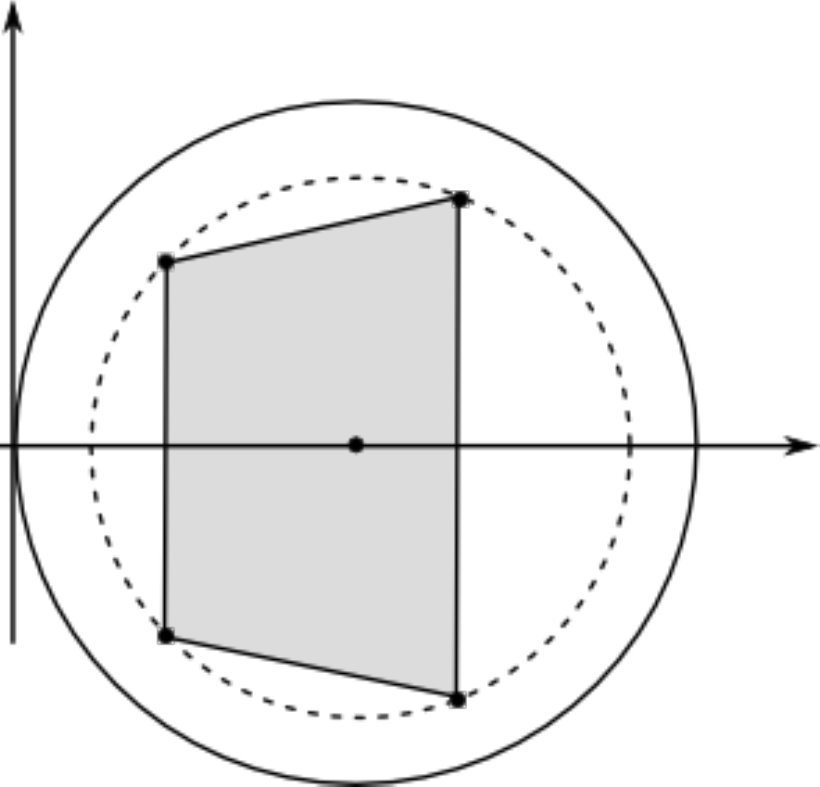}
\caption{The locus of minimizers co$(G)$ in Example \ref{noradial}}\label{fig: co_K_}
\end{center}
\begin{picture}(0,0)
\put(207,124){$\frac 12 I$}\put(225,51){$Q_1$}
\put(155,66){$Q_2$}
\put(157,162){$Q_3$}
\put(229,175){$Q_4$}
\put(203,195){$\partial K$}
\end{picture}
\end{figure}

\subsection{An explicit formula in three dimensions}\label{3d}
In this section we provide an explicit formula for the limiting energy density of a three-dimensional system. We consider nearest and next-to-nearest interactions on a cubic lattice for a special choice of the potentials. Our dual-lattice approach may be easily extended to the case of energy densities with four-point interactions of the type ${\bf f}(p,q,r,s)$ depending on the values that the microscopic vector field $u$ takes on the $4$ vertices of each face of the cubic cell and are invariant under permutations of the arguments. If our context, we consider ${\bf f}$ that can be written as a sum of two-point potentials. In this case, we obtain some relations between nearest and next-to-nearest neighbor interactions giving rise to energies as in \eqref{energy0thorder3d} below.

Following the scheme of Section \ref{2d} we consider the family of energies:
\begin{eqnarray}\label{energy0thorder3d}
F_\e(Q)=
\begin{cases}\displaystyle
\sum_{|i-j|=1}\e^3 f(u_i,\ u_j)+\frac{1}{4}\sum_{|i-j|=\sqrt{2}}\e^3 f(u_i,\ u_j)& \text{if $Q\in C_\e(\Om;K)$}\\
+\infty&\text{otherwise}.
\end{cases}
\end{eqnarray}
\begin{remark}
We observe that the choice of the next to nearest neighbor potentials as being exactly one fourth of the nearest-neighbor potential is crucial in deriving an explicit formula. It makes compatible the algebraic decomposition in Lemma \ref{decomp3} below with the topological structure of the graph given by the nearest and next-to-nearest bonds of the cubic lattice. 
\end{remark}
 \begin{lemma}\label{decomp3}
 Let $K$ be as in \eqref{convex} with $N=3$ and let $Q\in K$. Then there exist $u,v,w,z\in S^2$ such that 
 \begin{equation}\label{eqdecomp3}
 Q=\frac14 (u\otimes u+v\otimes v+w\otimes w+z\otimes z)
 \end{equation} 
 \end{lemma}
 \begin{proof}
 Let $0\leq\lambda_1\leq\lambda_2\leq\lambda_3$ be such that
 \begin{equation*}
 Q=\lambda_1 e_1\otimes e_1+\lambda_2 e_2\otimes e_2+\lambda_3 e_3\otimes e_3.
 \end{equation*}
From ${\rm tr}\,Q=1$ we get that $\frac12\leq\lambda_2+\lambda_3\leq 1$. We now set $\delta:=\lambda_2+\lambda_3-\frac12$ and observe that $\delta\leq\lambda_3$. In fact, assuming by contradiction $\delta>\lambda_3$ we would have $2(\lambda_2+\lambda_3)-1=2\delta>2\lambda_3\geq \lambda_2+\lambda_3$ that would imply $\lambda_2+\lambda_3>1$. Hence, we may define 
\begin{eqnarray}
u&=&\sqrt{2\lambda_2}e_2+\sqrt{2(\lambda_3-\delta)}e_3,\\
v&=&\sqrt{2\lambda_2}e_2-\sqrt{2(\lambda_3-\delta)}e_3,\\
w&=&\sqrt{2\lambda_1}e_1+\sqrt{2\delta}e_3,\\
z&=&\sqrt{2\lambda_1}e_1-\sqrt{2\delta}e_3.
\end{eqnarray}
Since, by the definition of $\delta$, we have that $2(\lambda_2+\lambda_3-\delta)=2(\lambda_1+\delta)=1$, it follows that $u,v,w,z\in S^2$ while a direct computation shows \eqref{eqdecomp3}.
 \end{proof}

We now set
\begin{equation}\label{fhat3d}
\widehat f(Q):=
\min\Bigl\{\sum_{\substack{i,j=1\\ i<j}}^4f(u_i,u_j):\ u_i\in S^2,\ \frac14\sum_{i=1}^4u_i\otimes u_i=Q\Bigr\}.
\end{equation}

\begin{theorem}\label{theorem:formula3d}
Let $F_\e\colon L^\infty(\Om;\Mtre)\to \R \cup\{\infty\}$ be defined by \eqref{energy0thorder3d}. Then $F_\e$ $\Gamma$-converges with respect to the weak$^*$-topology of $L^\infty(\Om;\Mtre)$ to the functional $F\colon L^\infty(\Om;\Mtre)\to \R \cup\{\infty\}$ defined by
$$
F(Q):=\begin{cases}
         \displaystyle
         \frac32\int_{\Om} \widehat f^{**}(Q(x))\,dx\quad\hbox{if }Q \in L^\infty(\Om; K)\\
         +\infty \quad\hbox{otherwise}.
         \end{cases}
$$
In particular $f_{\rm hom}=\frac32\widehat f^{**}$.
\end{theorem}
\begin{proof}
Let $Q\in L^{\infty}(\Omega;K)$ and let $Q_\e\in C_\e(\Omega;K)$ be such that $Q_\e\wtos Q$. For all $i\in \Z_\e(\Omega)$ we set $u_i$ such that $Q(\e i)=Q(u_i)$ and define the dual cells $P_{i}^{kl}$ (see Figure \ref{fig: dual_3d}) as follows
\begin{equation}
P_{i}^{kl}={\rm co}\Bigl\{S_i^{kl},i+\frac{1}{2}(e_k+e_l)+\frac{1}{2}(e_k\wedge e_l),i+\frac{1}{2}(e_k+e_l)-\frac{1}{2}(e_k\wedge e_l)\Bigr\}
\end{equation}
where $e_k$ and $e_l$ are two distinct vectors of the canonical base of $\R^3$ and $S_i^{kl}$ is the unitary square in the plane spanned by them having $i$ as left bottom corner; namely, 
\begin{equation}
S_i^{kl}={\rm co}\{i,i+e_k,i+e_k+e_l,i+e_l\}. 
\end{equation}
Observe that $|P_i^{kl}|=\frac13$. To $Q_\e$ we associate the dual piecewise-constant interpolation $Q_\e'$ defined as 
\begin{equation}\label{Qprimo}
Q_\e'(x)=\frac14\left(Q_\e(\e i)+Q_\e(\e (i+e_k))+Q_\e(\e(i+e_k+e_l))+Q_\e(\e (i+e_l))\right)\text{for all }x\in \e P_{i}^{kl}.
\end{equation}
First, we show that $Q_\e-Q'_\e\wtos 0$. Let us fix $\Omega'\subset\subset\Omega$ and $A\in \Mtre$. We have that
\begin{eqnarray*}
&&\int_{\Omega'}A:Q_\e'(x)\, dx=\sum_{i\in\Z_\e(\Omega')}\sum_{\substack{k,l=1\\ k<l}}^3\int_{\e P_i^{kl}} A:Q_\e'(x)\, dx+o(1)\\
&=&\sum_{i\in\Z_\e(\Omega')}\sum_{\substack{k,l=1\\ k<l}}^3\e^3A:\frac{1}{12}\Bigl(Q(\e i)+Q_\e(\e (i+e_k))+Q_\e(\e(i+e_k+e_l))+Q_\e(\e (i+e_l))\Bigr)+o(1)\\
&=&\sum_{j\in\Z_\e(\Omega')}\e^3A:Q_\e(\e j)+o(1)=\int_{\Omega'}A:Q_\e(x)\, dx+o(1).
\end{eqnarray*} 
The last equality is obtained by reordering the sums observing that each $j\in\Z_\e(\Omega')$ appears exactly $12$ times ($4$ for each of the three possible choices of $k<l$). By the arbitrariness of $A$ and $\Omega'$ it follows that $Q_\e-Q'_\e\wtos 0$. 
We now prove the liminf inequality. We may write 
\begin{eqnarray*}
F_\e(Q_\e)&=&2\ \frac14\sum_{i\in\Z_\e(\Omega')}\sum_{\substack{k,l=1\\ k<l}}^3\e^3\Bigl[f(u(\e i),u(\e (i+e_k)))+f(u(\e i),u(\e (i+e_l)))+ \nonumber\\ &&f(u(\e (i+e_k)),u(\e (i+e_k+e_l)))+f(u(\e (i+e_l)),u(\e (i+e_k+e_l)))+ \nonumber\\&& \vphantom{\sum_{i\in\Z_\e(\Omega')}\sum_{\substack{k,l=1\\ k<l}}^3\e^3}f(u(\e i),u(\e (i+e_k+e_l)))+f(u(\e (i+e_k)),u(\e (i+e_l)))\Bigr]+o(1),
\end{eqnarray*}
where the prefactor $2$ appears since we are passing from an unordered sum to an ordered one, while the additional $1/4$ in front of the nearest-neighbor interaction potentials is due to the fact that each nearest-neighboring pair (apart from those close to the boundary which carry an asymptotically negligible energy) corresponds to $4$ distinct choices of the indices $i,k,l$.  Taking into account \eqref{fhat3d} we continue the above estimate as
\begin{eqnarray*}
F_\e(Q_\e)&\geq& \frac12\sum_{i\in\Z_\e(\Omega')}\sum_{\substack{k,l=1\\ k<l}}^3\e^3\widehat f\Bigl(Q'_\e\Bigl(\e i+\frac\e2(e_k+e_l)\Bigr)\Bigr)+o(1)\\&=&\frac32 \sum_{i\in\Z_\e(\Omega')}\sum_{\substack{k,l=1\\ k<l}}^3\int_{\e P_i^{kl}}\widehat f(Q'_\e(x))\ dx+o(1)\\
&\geq&\frac32\int_\Omega\widehat f^{**}(Q'_\e(x))\ dx+o(1)
\end{eqnarray*}
The liminf inequality follows passing to the liminf as $\e\to 0$. 

We now prove the limsup inequality. By Theorem \ref{32-th:homog-Li} it suffices to prove that for any constant $Q\in K$ and every open set $A\subset\Omega$ it exists a sequence $Q_\e\wtos Q$ such that 
\begin{equation}\label{claim_limsup}
\limsup_\e F_\e(Q_\e, A)\leq\frac32\widehat f(Q)|A|.
\end{equation}
In fact, by formula \eqref{Fhom}, the arbitrariness of $Q$ and $A$ and the convexity of $f_{\rm hom}$, this would imply that 
\begin{equation*}
f_{\rm hom}(Q)\leq \frac32\widehat f^{**}(Q)
\end{equation*}
thus concluding the proof of the limsup inequality. By the locality of the construction we will prove \eqref{claim_limsup} in the case $A=(-l,l)^3$, where, up to a translation, we are supposing that $0\in\Omega$. 
Let $p,q,r,s$ realize the minimum in formula \eqref{fhat3d} for the given $Q$. We set $u(0)=p$ and construct a 2-periodic function $u$ whose unit cell is pictured in Figure \ref{optimal_3d}. We then set $u_\e(\e\alpha)=u(\alpha)$
and $Q_\e=u_\e\otimes u_\e$.
%and then we follow the scheme below (see Figure \ref{optimal_3d}): for all $i\in\Z_\e(\Omega)$
%\begin{eqnarray*}
%u_\e(\e i)=p\iff u(\e (i\pm e_1))=q\quad && u_\e(\e i)=r\iff u(\e (i\pm e_1))=s\\
%u_\e(\e i)=p\iff u(\e (i\pm e_2))=r\quad && u_\e(\e i)=q\iff u(\e (i\pm e_2))=s\\
%u_\e(\e i)=p\iff u(\e (i\pm e_3))=s\quad && u_\e(\e i)=q\iff u(\e (i\pm e_3))=r,
%\end{eqnarray*}
Note that for any choice of $i\in\Z_\e(\Omega)$ and $k,l\in\{1,2,3\}$ with $k<l$ the values of $u_\e$ on the vertices of $S_i^{kl}$ are always all the $4$ values $p,q,r,s$. As a result the dual interpolation $Q'_\e$ constructed as in \eqref{Qprimo} is constantly equal to $Q$. Since $Q_\e'-Q_\e\wtos 0$ this gives that $Q_\e\wtos Q$. Furthermore the following equality holds
\begin{eqnarray}
F_\e(Q_\e,A)&=&2\ \frac14\sum_{i\in\Z_\e(A)}\sum_{\substack{k,l=1\\ k<l}}^3\e^3[f(p,q)+f(r,s)+f(p,r)+f(q,s)+ f(p,s)+f(q,r)]+o(1)\nonumber\\ &=&\frac32 \sum_{i\in\Z_\e(A)}\sum_{\substack{k,l=1\\ k<l}}^3|\e P_i^{kl}|\widehat f(Q)+o(1)=\frac32\widehat f(Q)|A|+o(1)
\end{eqnarray}
which implies \eqref{claim_limsup}.
\end{proof}

\begin{figure}
\begin{center}
\includegraphics[scale=.5 ]{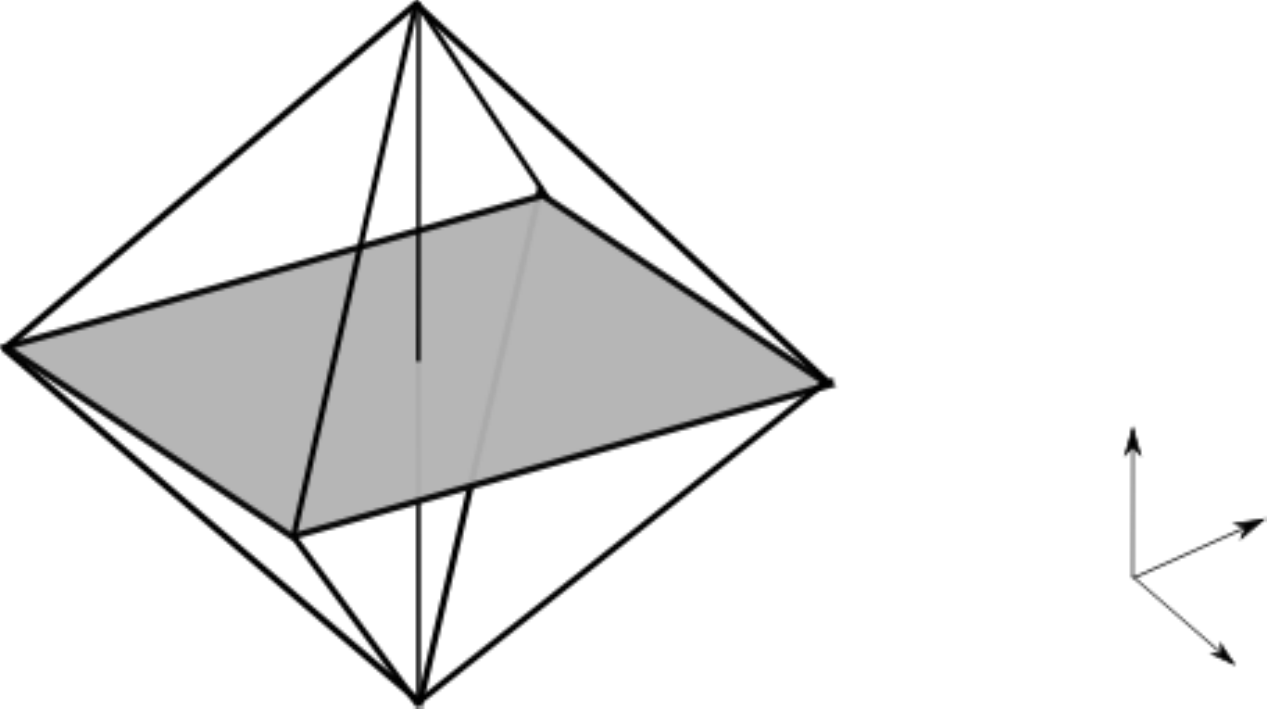}
\caption{The dual cell $P_i^{kl}$ with $S_i^{kl}$ in grey}\label{fig: dual_3d}
\end{center}
\begin{picture}(0,0)
\put(123,94){$i$}\put(301,43){$e_k$}
\put(306,73){$e_l$}
\end{picture}
\end{figure}
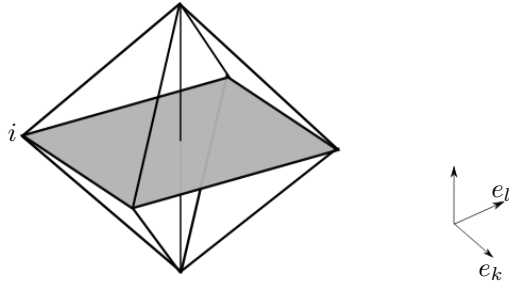

\begin{figure}
\begin{center}
\includegraphics[scale=.5 ]{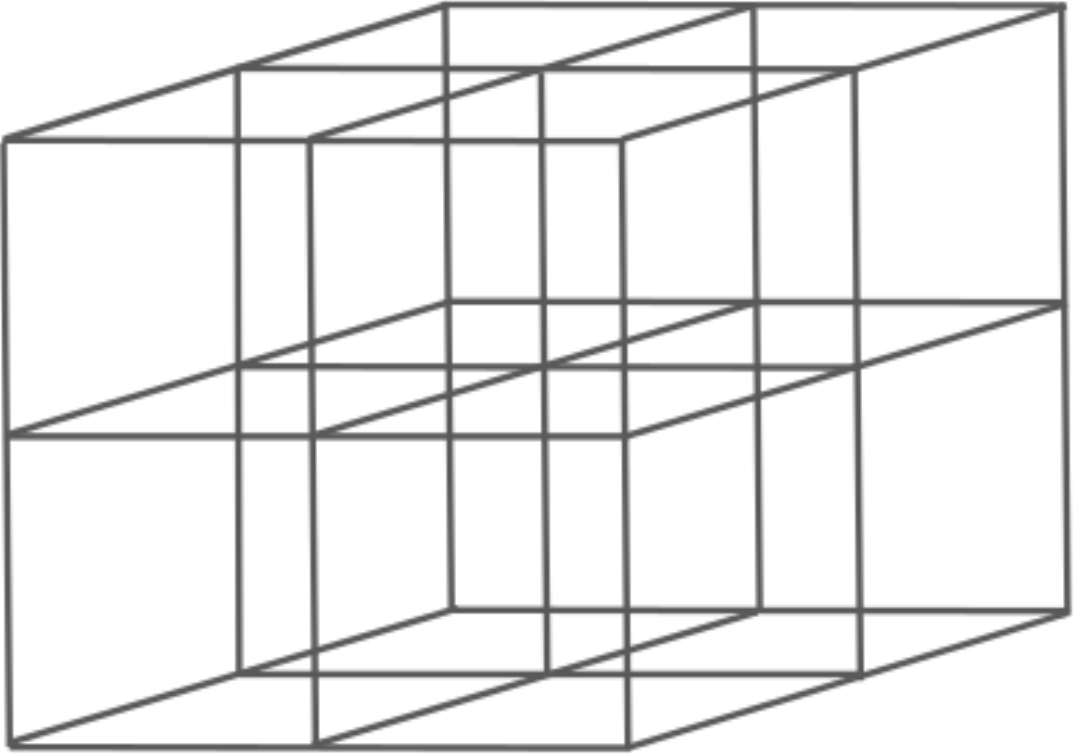}
\caption{The recovery sequence in Theorem \ref{theorem:formula3d} on a periodicity cell}
\label{optimal_3d}
\end{center}
\begin{picture}(0,0)
\put(216,103){{\footnotesize $p$}}\put(267,103){{\footnotesize $q$}}\put(168,103){{\footnotesize $q$}}
\put(181,93){{\footnotesize $r$}}\put(252,110){{\footnotesize $r$}}
\put(216,145){{\footnotesize $s$}}\put(216,58){{\footnotesize $s$}}
\put(263,145){{\footnotesize $r$}}\put(171,145){{\footnotesize $r$}}
\put(266,58){{\footnotesize $r$}}\put(169,58){{\footnotesize $r$}}
\put(182,39){{\footnotesize $q$}}\put(252,67){{\footnotesize $q$}}
\put(182,136){{\footnotesize $q$}}\put(247,155){{\footnotesize $q$}}
\put(237,88){{\footnotesize $s$}}\put(296,110){{\footnotesize $s$}}
\put(138,93){{\footnotesize $s$}}\put(202,110){{\footnotesize $s$}}
\put(228,136){{\footnotesize $p$}}\put(291,155){{\footnotesize $p$}}
\put(135,136){{\footnotesize $p$}}\put(201,155){{\footnotesize $p$}}
\put(225,39){{\footnotesize $p$}}\put(292,67){{\footnotesize $p$}}
\put(135,39){{\footnotesize $p$}}\put(203,67){{\footnotesize $p$}}
\end{picture}
\end{figure}
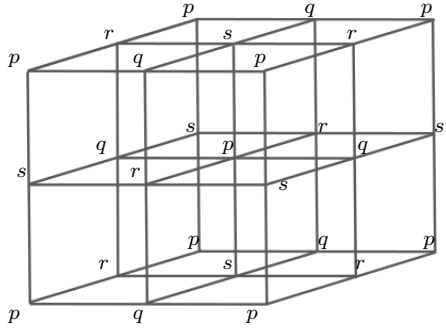

\section{Gradient-type and concentration scalings}\label{sec:grad}
In this section we assume that $\Om\subset\R^2$ is a simply connected set. For all $s\in (0,1]$ we define the subset $\partial K_s$ of $K$ as 
\begin{equation}\label{partialK_s}
\partial K_s:=\Bigl\{Q\in K,\, \Bigl|Q-\frac{1}{2}I\Bigr|=\frac{\sqrt{2}}{2}s\Bigr\}.
\end{equation}
In the case $s=1$ we omit the subscript $s$.; in that case we have
\begin{equation}\label{partialK}
\partial K=\{Q\in K,\, |Q|=1\}.
\end{equation}
\begin{remark}\label{remark:approx}
The simple-connectedness assumption on $\Om$ plays a crucial role in the following argument: since by \eqref{KN2} $K$ can be identified with a ball in the two-dimensional affine space of the matrices $Q\in\Mdue$ with ${\rm tr}\,Q=1$ and consequently $\partial K_s$ can be identified with $S^1$, by the simple connectedness of $\Omega$ we may apply Theorem $3$ in \cite{BetZhe} stating that $C^{\infty}(\Om;\partial K_s)\cap W^{1,2}(\Om;\partial K_s)$ is dense in the space of $W^{1,2}(\Om;\partial K_s)$ strongly in $W^{1,2}(\Om;\partial K_s)$. This ensures that any given $Q\in W^{1,2}(\Om;\partial K_s)$ is orientable; that is, there exists $n\in W^{1,2}(\Om;S^1)$ such that 
\begin{equation}\label{lifting}
Q(x)=s\Bigl(n(x)\otimes n(x)-\frac12 I\Bigr)+\frac12 I.
\end{equation}
This has the following two consequences. First of all, setting $n^\perp (x)$ the counterclockwise rotation of $n(x)$ by $\frac \pi2$, we have that 
\begin{equation}
Q(x)=\frac{1+s}{2}n(x)\otimes n(x)+\frac{1-s}{2}n^\perp(x)\otimes n^\perp(x);
\end{equation}  
hence, setting 
\begin{eqnarray}\label{decomposizione}
v(x)=\sqrt{\frac{1+s}{2}}n(x)+\sqrt{\frac{1-s}{2}}n^\perp(x),\nonumber\\ \\
w(x)=\sqrt{\frac{1+s}{2}}n(x)-\sqrt{\frac{1-s}{2}}n^\perp(x), \nonumber
\end{eqnarray}
one has that 
\begin{equation}\label{formula_magica_2}
Q(x)=\frac12v(x)\otimes v(x)+\frac12w(x)\otimes w(x),
\end{equation}
where $v,w\in W^{1,2}(\Om; S^1)$. If $Q\in C^{\infty}(\Om;\partial K_s)\cap W^{1,2}(\Om;\partial K_s)$ we moreover have that $v,w\in C^{\infty}(\Om; S^1)\cap W^{1,2}(\Om; S^1)$.
Finally, if $Q\in W^{1,2}(\Om; \partial K_s)$ and $n \in W^{1,2}(\Om; S^1)$ are related by \eqref{lifting}, by \eqref{fattoregiusto} and a density argument one gets
\begin{equation}\label{fattoregiusto2}
|\nabla Q(x)|^2= 2 s^2|\nabla n(x)|^2
\end{equation}
for a.e.\ $x\in \Om$.
\end{remark}

\subsection{Sobolev scaling - Selection of uniform states}\label{uniform states}

In this section we consider higher-order scalings of energies of the form:
\begin{equation}\label{Eeps}
E_\e(u)=\sum_{|i-j|=1}\e^2h(|(u_i\cdot u_j)|)
\end{equation} 
where $h$ is a bounded Borel function. Via the usual identification $Q(\e i)=Q(u_i)$ we may associate to $E_\e(u)$ the functional $F_\e(Q)$ as done in the previous section. We now scale $F_\e$ by $\e^2$ as follows: 
\begin{equation*}
F_\e^1(Q):=\frac{F_\e(Q)-\inf F_\e}{\e^2}.
\end{equation*}
As usual we extend this functional (without renaming it) and consider
\begin{eqnarray}\label{energy1storder}
F^1_\e(Q)=
\begin{cases}
\sum_{|i-j|=1}h(|(u_i\cdot u_j)|)-\inf h& \text{if $Q\in C_\e(\Om;K)$}\\
+\infty&\text{otherwise}.
\end{cases}
\end{eqnarray}
A relevant example of an energy of the type above is the one in the Lebwohl-Lasher model  (see \cite{CDSZ}, \cite{LebLas}), which corresponds to the case $h(x)=-x^2$.

We consider the case when  $\inf h=h(1)$, or equivalently the function $\widehat f$ in \eqref{grt} attains its minimum on all points of $\partial K$. In this case, by Remark \ref{r: h++}, the zero-th order $\Gamma$-limit of $\e^2 F^1_\e$ is identically $0$. We will additionally assume that $h\in C^1([0,1])$ and that there exists $\delta>0$ such that $h \in C^2([1-\delta, 1])$. Under these hypotheses we are able to estimate the $\Gamma$-limsup of $F^1_\e$. In order to estimate the $\Gamma$-liminf, we will make the following assumption on $h$: there exists $\gamma>0$ such that
\begin{equation}\label{hypo}
h(x)-h(1)\ge \frac{\gamma}{2}(1-x^2)
\end{equation}
for every $x \in [0,1]$.

\begin{remark}
Hypothesis \eqref{hypo} implies in particular that $1$ is the unique absolute minimum for $h$ in $[0,1]$. Since a function $\gamma$ as in \eqref{hypo} must satisfy $\gamma \le |h'(1)|$, when such an hypothesis holds, we have $h'(1) <0$. If $h'(1)=0$, so that in particular \eqref{hypo} cannot hold, we will show later that the $\Gamma$-liminf lower bound of Theorem \ref{gradterm} is not true, even for convex $h$. Namely, in this degenerate case the $\Gamma$-liminf can be finite also on functions whose gradient is not in $L^2$. This will be shown in the example at the end of this section.
Examples of functions satisfying \eqref{hypo} are all convex functions on $[0,1]$ with $h'(1) <0$. In this case, indeed, by convexity and since $h'(1)<0$ one has
$$
h(x)-h(1)\ge h'(1)(x-1)=|h'(1)|(1-x)
$$
for every $x \in [0,1]$. By means of the elementary inequality $1-x \ge \frac12 (1-x^2)$ for every $x \in [0,1]$, one gets \eqref{hypo} with $\gamma$ that can be taken exactly equal to $|h'(1)|$. For such energy densities, the full $\Gamma$-convergence result of Theorem \ref{gradterm}, part (c), holds.

Among nonconvex functions, \eqref{hypo} is for instance satisfied in the case $h(x)=-x^p$ with $p\ge 1$. In this case, $\gamma=\min\{2,p\}$. In particular, if $1\le p\le 2$, then $\gamma=p=|h'(1)|$ and the full $\Gamma$-convergence result again holds. Otherwise, we only have a lower bound on the $\Gamma$-liminf of $F^1_\e$ which is of the same type of the upper bound on the $\Gamma$-limsup, but with a different constant multiplying the Dirichlet integral.
\end{remark}

We will prove the following result.
 
\begin{theorem}\label{gradterm}
Let $F^1_\e\colon L^\infty(\Om;\Mdue)\to \R \cup\{\infty\}$ be defined by \eqref{energy1storder} with $h\in C^1([0,1])$.
Assume that there exists $\delta>0$ such that $h \in C^2([1-\delta, 1])$. Define the functional $F^1\colon L^\infty(\Om;\Mdue)\to \R \cup\{\infty\}$ as
$$
F^1(Q):=\begin{cases}
         \displaystyle
         \frac{|h'(1)|}{2}\int_{\Om} |\nabla Q(x)|^2\,dx\quad\hbox{if }Q \in W^{1,2}(\Om; \partial K)\\
         +\infty \quad\hbox{otherwise}
         \end{cases}
$$
with $\partial K$ as in \eqref{partialK}.

{\rm(a)} Let $(F^1)''$ be the $\Gamma$-limsup of  $F^1_\e$ with respect to the weak$^*$-topology of $L^\infty(\Om;\Mdue)$. Then $(F^1)''\le F^1$.

{\rm(b)} Assume that in addition \eqref{hypo} holds, and define the functional $F^1_\gamma \colon L^\infty(\Om;\Mdue)\to \R \cup\{\infty\}$ by
$$
F^1_\gamma(Q):=\begin{cases}
         \displaystyle
         \frac{\gamma}{2}\int_{\Om} |\nabla Q(x)|^2\,dx\quad\hbox{if }Q \in W^{1,2}(\Om; \partial K)\\
         +\infty \quad\hbox{otherwise}\,.
         \end{cases}
$$
Denote by $(F^1)'$ the $\Gamma$-liminf of  $F^1_\e$ with respect to the weak$^*$-topology of $L^\infty(\Om;\Mdue)$. Then $(F^1)'\ge F^1_\gamma$.

{\rm ({c})} If in particular we can take $\gamma=|h'(1)|$ in \eqref{hypo}, then $F^1_\e$ $\Gamma$-converges with respect to the weak$^*$-topology of $L^\infty(\Om;\Mdue)$ to the functional $F^1$.
\end{theorem}

The following lemma will be useful in the proof.
\begin{lemma}\label{Lemma_strong_conv}
Let $\Om$ be an open subset of $\R^2$. Given a function $
u\colon \e\Z_\e(\Om) \to S^{1}$ let $Q_{\e}$ be the piecewise-constant interpolation of $Q(u_\e)$ and let $Q_\e^a$ be the piecewise-affine interpolation of $Q(u_\e)$ having constant gradient on triangles with vertices $\e i$ and longest side parallel to $e_1-e_2$. Then
\begin{equation}\label{interpol}
\int_{\Om_\e}|Q_\e(x)-Q^a_\e(x)|^2\,dx \le \frac{1}{2}\e^2 \int_{\Om_\e} |\nabla Q^a_\e(x)|^2\,dx 
\end{equation}
In particular, if $\int_{\Om_\e} |\nabla Q^a_\e(x)|^2\,dx$ is uniformly bounded, then $Q_\e$ is $L^2$-compact as $\e \to 0$, and each limit point $Q$ of $Q_\e$ belongs to $W^{1,2}(\Om;\partial K)$.
\end{lemma}

\begin{proof}
Let $i\in \Z^2$ be such that $\e \{i+W\} \subset \Om$ and fix $x$ in the interior of such a cube. By construction, up to a null set the gradient $\nabla Q^a_\e$ is constant on the segment joining $x$ and $\e i$, the only possible exception being  when $x-\e i$ is parallel to one of the coordinate axes, or to $e_1-e_2$. By this and the mean value theorem we then have that 
$$
|Q_\e(x)-Q^a_\e(x)|=|Q^a_\e(\e i)-Q^a_\e(x)|\le|x-\e i||\nabla Q^a_\e(x)|
$$
for a.e.\ $x \in \e \{i+W\}$. Therefore
$$
|Q_\e(x)-Q^a_\e(x)|^2\le\frac{1}{2}\e^2 |\nabla Q^a_\e(x)|^2
$$
for a.e.\ $x \in \e \{i+W\}$. Summing over all such cubes, we get \eqref{interpol}.

Since $Q^a_\e$ are uniformly bounded in $L^\infty$, if $\int_{\Om_\e} |\nabla Q^a_\e(x)|^2\,dx$ is uniformly bounded, the sequence $Q^a_\e$ is $L^2$-compact by the Rellich Theorem and the equiintegrability given by the uniform bound $|Q^a_\e(x)|\le 1$. So is then $Q_\e$ by \eqref{interpol}, and it has the same limit points. Each limit point belongs then to $L^2(\Om;\partial K)$ since this set is closed with respect to strong $L^2$ convergence. By the boundedness of $\int_\Om |\nabla Q^a_\e(x)|^2\,dx$, we also infer that actually each limit point $Q$ of $Q_\e$ belongs to $W^{1,2}(\Om;\partial K)$.
\end{proof}

We are now in a position to give the proof of Theorem \ref{gradterm}.
\begin{proof}[{\it Proof of Theorem {\rm\ref{gradterm}:}}]
(a) We need to prove the inequality only for $Q\in W^{1,2}(\Omega; \partial K)$. Let $R>0$ be such that $\Om\subset RW$ and let $Q\in W^{1,2}(RW;\partial K)$ denote the (not renamed) extension of any $Q\in W^{1,2}(\Om;\partial K)$ which exists thanks to the regularity of $\Om$. As observed in Remark \ref{remark:approx}, the space of matrices of the type $Q(x)=u(x)\otimes u(x)$ with $u \in C^\infty(RW; S^1)$ is dense in $W^{1,2}(RW;\partial K)$. Thus, it suffices to prove the upper-bound inequality only for such $Q$. 
For every $\e$, let now $u_\e(x)$ be the piecewise-affine approximation of $u(x)$ such that $u_\e(\e i)=u(\e i)$ for every $i \in \Z_\e(RW)$, and set $Q^a_\e(x)=u_\e(x)\otimes u_\e(x)$. We have
\begin{equation}\label{conv}
F^1(Q)=\lim_{\e \to 0}\frac{|h'(1)|}{2}\int_{\Om +B(0,\e)} |\nabla Q^a_\e(x)|^2\,dx\,.
\end{equation}

By the regularity of $u$, the functions $u_\e$ are bounded in $W^{1,\infty}$, which means that there exists a constant $M$ independent of $\e$ such that
$$
|u_\e(\e i)-u_\e(\e j)|\le 2M \e
$$
for every $i,j \in \Z_\e(RW)$ with $|i-j|=1$. This implies, by a direct computation, that
\begin{equation}\label{bdd}
|u_\e(\e i)\cdot u_\e(\e j)| \ge 1-M \e^2.
\end{equation}
In particular, when $\e$ is small enough, $|u_\e(\e i)\cdot u_\e(\e j)|$ belongs to the interval $[1-\delta, 1]$ where $h$ is $C^2$.

Using \eqref{prodotti} and \eqref{modul} we have
\begin{eqnarray}\label{fine}
\int_{\Om+B(0,\e)} \!\!|\nabla Q^a_\e(x)|^2\,dx&\geq& \nonumber \!\!\!\sum_{\substack{i,j \in \Z_\e(\Om)\\|i-j|=1}}{\frac 12 \e^2 \Big|\frac{u_\e(\e i)\otimes u_\e(\e i)-u_\e(\e j)\otimes u_\e(\e j)}{\e}\Big|^2}\\&=&\!\!\!\sum_{\substack{i,j \in \Z_\e(\Om)\\|i-j|=1}}{(1-|u_\e(\e i)\cdot u_\e(\e j)|^2)},
\end{eqnarray}
where we have taken into account that every triangular cell has measure $\e^2/2$ and that every interaction between $i$ and $j$ belonging to $\Z_\e(\Om)$ appears with the same factor two, both in the sum (since it is not ordered) and in the integral. 
Observe that by \eqref{bdd} one has
\begin{eqnarray}\label{estim}\nonumber
1-|u_\e(\e i)\cdot u_\e(\e j)|\le M\e^2,\quad 1-|u_\e(\e i)\cdot u_\e(\e j)|^2\le 2M\e^2,\\
\\
\nonumber \frac12\ge \frac{1}{1+|u_\e(\e i)\cdot u_\e(\e j)|}-\frac M2 \e^2,
\end{eqnarray}
so that inserting these inequalities in the previous estimate we arrive to
\begin{eqnarray}\label{inew}\nonumber
\int_{\Om+B(0,\e)} \!\!|\nabla Q^a_\e(x)|^2\,dx&\ge& 2\sum_{\substack{i,j \in \Z_\e(\Om)\\|i-j|=1}}{\frac {1-|u_\e(\e i)\cdot u_\e(\e j)|^2}{1+|u_\e(\e i)\cdot u_\e(\e j)|}}-M \e^2\sum_{\substack{i,j \in \Z_\e(\Om)\\|i-j|=1}}{(1-|u_\e(\e i)\cdot u_\e(\e j)|^2)}\\
\\ \nonumber
&\ge&2\sum_{\substack{i,j \in \Z_\e(\Om)\\|i-j|=1}}{(1-|u_\e(\e i)\cdot u_\e(\e j)|)}-4M^2\e^2|\Om|.
\end{eqnarray}

If now $C$ denotes an upper bound for $h''/2$ in $[1-\delta, 1]$, we have
$$
h(x)-h(1)\le |h'(1)|(1-x)+C (1-x)^2
$$
for every $x \in [1-\delta, 1]$. Multiplying \eqref{inew} by $|h'(1)|$ and inserting this last inequality, we get
\begin{eqnarray*}
|h'(1)|\int_{\Om+B(0,\e)} \!\!|\nabla Q^a_\e(x)|^2\,dx
&\ge& 2\sum_{\substack{i,j \in \Z_\e(\Om)\\ |i-j|=1}} (h(|u_\e(\e i)\cdot u_\e(\e j)|)-h(1))
\\
&&-8|h'(1)|M^2 \e^2 |\Om|-2\ C \sum_{\substack{i,j \in \Z_\e(\Om)\\|i-j|=1}} (1-|u_\e(\e i)\cdot u_\e(\e j)|)^2,
\end{eqnarray*}
whence, using \eqref{estim}, we deduce that there is a positive constant $L$ independent of $\e$ such that 
$$
|h'(1)|\int_{\Om+B(0,\e)} \!\!|\nabla Q^a_\e(x)|^2\,dx\ge 2\sum_{\substack{i,j \in \Z_\e(\Om)\\|i-j|=1}}{(h(|u_\e(\e i)\cdot u_\e(\e j)|)-h(1))}-L\e^2\,.
$$
Defining now $Q_\e$ the piecewise-constant interpolation of $Q(u_\e)$, by Lemma \ref{Lemma_strong_conv}, we have that $Q_\e\to Q$ strongly in $L^2(\Om;\partial K)$. We may rewrite the previous inequality as
$$
\frac{|h'(1)|}{2}\int_{\Om+B(0,\e)} \!\!|\nabla Q^a_\e(x)|^2\,dx\ge F^1_\e(Q_\e)-L\e^2\,.
$$
Taking the $\limsup$ as $\e\to 0$, by \eqref{conv} we deduce the upper-bound inequality.

(b) We denote with $(F^1)'$ the $\Gamma$-liminf of $F^1$. We want to show the $\Gamma$-liminf inequality $(F^1)'\ge F^1$. Let $Q_\e \in C_\e(\Om;K)$ be the sequence of piecewise-constant functions associated to $Q(u_\e)$, for some $u_\e\colon \e\Z^2 \cap \Om \to S^{1}$ and let $Q^a_\e$ be the piecewise-affine interpolations satisfying $Q^a_\e(\e i)=u_\e(\e i)\otimes u_\e(\e i)$ for every $i \in \Z_\e(\Om)$. We claim that it suffices to prove
\begin{equation}\label{claim_0}
\sum_{\substack{i,j \in \Z_\e(\Om)\\|i-j|=1}}{[h(|u_\e(\e i)\cdot u_\e(\e j)|)-h(1)]}\ge \frac{\gamma}{2} \int_{\Om_\e} |\nabla Q^a_\e(x)|^2\,dx \,.
\end{equation}
Indeed by \eqref{energy1storder}, \eqref{claim_0} is equivalent to 
$$
F^1_\e(Q_\e)\ge \frac{\gamma}{2}\int_{\Om_\e} |\nabla Q^a_\e(x)|^2\,dx \,;
$$
therefore, if the left-hand side keeps bounded, by Lemma \ref{Lemma_strong_conv} we have that $Q_\e-Q^a_\e \to 0$ strongly in $L^2(\Om;K)$ and each limit point of $Q_\e$ must belong to $W^{1,2}(\Om;\partial K)$. Furthermore, the inequality $(F^1)'(Q)\ge F^1(Q)$ for $Q\in W^{1,2}(\Om;\partial K)$ follows from \eqref{claim_0} by semicontinuity.

Now, since \eqref{hypo} gives
$$
\sum_{\substack{i,j \in \Z_\e(\Om)\\|i-j|=1}}{[h(|u_\e(\e i)\cdot u_\e(\e j)|)-h(1)]}\ge \frac{\gamma}{2} \sum_{\substack{i,j \in \Z_\e(\Om)\\|i-j|=1}}{ (1-|u_\e(\e i)\cdot u_\e(\e j)|^2)}
$$
\eqref{claim_0} immediately follows from the inequality 
\begin{eqnarray*}
\sum_{\substack{i,j \in \Z_\e(\Om)\\|i-j|=1}}{ (1-|u_\e(\e i)\cdot u_\e(\e j)|^2)}\geq\int_{\Om_\e} |\nabla Q^a_\e(x)|^2\,dx \,,
\end{eqnarray*}
that can be obtained from 
\begin{eqnarray*}
\int_{\Om_\e} \!\!|\nabla Q^a_\e(x)|^2\,dx&\leq& \!\!\!\sum_{\substack{i,j \in \Z_\e(\Om)\\|i-j|=1}}{\frac 12 \e^2 \Big|\frac{u_\e(\e i)\otimes u_\e(\e i)-u_\e(\e j)\otimes u_\e(\e j)}{\e}\Big|^2}\\&=&\!\!\!\sum_{\substack{i,j \in \Z_\e(\Om)\\|i-j|=1}}{(1-|u_\e(\e i)\cdot u_\e(\e j)|^2)}
\end{eqnarray*}
which holds by construction of $Q_\e^a$.
Finally, (c) is an obvious consequence of (a) and (b).
\end{proof}

\medskip
At the end of this section we give an example of an energy of the type \eqref{Eeps} such that $\inf h=h(1)=0$ and that $h'(1)=0$. In this case the assumptions of Theorem \ref{gradterm} are not satisfied and indeed we can show that the domain of the $\Gamma$-$\limsup$ of the $\e^2$-scaled energy \eqref{energy1storder} is strictly larger than $W^{1,2}(\Om;\partial K)$.  

\begin{example}\rm Let $h(x)=(1-x)^2$ so that $F_\e^1$ takes the form:
\begin{eqnarray}\label{energy1storder-example}
F^1_\e(Q)=
\begin{cases}\displaystyle
\sum_{|i-j|=1}(1-|(u_i\cdot u_j)|)^2&\text{if $ Q\in C_\e(\Om;K)$}\\
+\infty&\text{otherwise}.
\end{cases}
\end{eqnarray}
Consider $Q(x)=\frac{x}{|x|}\otimes \frac{x}{|x|}$. Note that $Q\notin W^{1,2}(\Om;\partial K)$ (while $Q\in W^{1,p}(\Om;\partial K)$ for all $1\leq p<2$). Setting $u(x)=\frac{x}{|x|}$ we have that $Q(x)=Q(u(x))$. To show that $\Gamma\hbox{-}\limsup_\e F_\e^1(Q)<+\infty$, it suffices to construct a sequence of unitary vector fields $u_\e\to u$ in $L^1$ (this implies that $Q_\e\to Q$ in $L^1$) such that $\limsup_\e\frac{E_\e(u_\e)}{\e^2}<+\infty$. We give an arbitrary value $u_0 \in S^1$ to the function $u$ in the origin, and define $u_\e$ as the piecewise-affine interpolation of $u(x)$. As usual, we construct this interpolation using as elements triangles with integer vertices and longest side parallel to $e_1-e_2$. We have that $u_\e\to u$ in $L^1$. 
We first observe that we have
$$
\frac{1}{\e^2}E_\e(u_\e)\le 8+\sum_{\substack{i,j \in \Z_\e(\Om)\setminus\{(0,0)\}\\|i-j|=1}}(1-|(u_i\cdot u_j)|)^2\,;
$$
this is obtained by estimating with $1$ all the $4$ interactions (each appearing twice) between $0$ and its nearest neighbors.
Taking into account that for all $i,\ j\in \Z_\e(\Om)\setminus\{(0,0)\}$, $|i-j|=1$ we have that $(u_\e(\e i)\cdot u_\e(\e j))=|(u_\e(\e i)\cdot u_\e(\e j))|$ we may write 
\begin{equation*}
\frac{1}{\e^2}E_\e(u_\e)\le 8+\sum_{\substack{i,j \in \Z_\e(\Om)\setminus\{(0,0)\}\\|i-j|=1}}(1-(u_i\cdot u_j))^2.
\end{equation*}
Observing that
$$
(1-(u_i\cdot u_j))^2=\frac 14|u_i-u_j|^4
$$
we are only left to show that
\begin{equation}\label{example:claim}
\limsup_\e\,\frac 14\sum_{\substack{i,j \in \Z_\e(\Om)\setminus\{(0,0)\}\\|i-j|=1}}|u_\e(\e i)-u_\e (\e j)|^4 <+\infty\,.
\end{equation}

In order to prove the claim, we make use of the following simple inequality (whose proof is omitted): for all $x$ und $y\in \R^2$, with $|x|\ge \e$, $|y|\ge \e$ and $|y-x|\le \sqrt{2}{\e}$ one has
\begin{equation}\label{ex:ineq}
\frac{1}{|y|^4}\le \frac{(\sqrt{2}+1)^4}{|x|^4}\,.
\end{equation}
Consider now a triangle $T$ of the interpolation grid, satisfying $T\subset \subset \R^2\setminus 2\e W$: in this way, however taken $x$ and $y \in T$, \eqref{ex:ineq} is satisfied. Let $\e i_1$, $\e i_2$, and $\e i_3 \in \Z_\e(\Om)\setminus\{(0,0)\}$ be the vertices of $T$, ordered in a way that $|i_1-i_2|=|i_1-i_3|=1$. By the Mean Value Theorem, one gets the existence of two points $v_{12}$ and $v_{13}$ both belonging to $T$, such that
$$
|u_\e (\e i_1)- u_\e (\e i_2)|^4=|u(\e i_1)- u (\e i_2)|^4=\e^4 |\nabla u(v_{12})|^4=\frac{\e^4}{|v_{12}|^4}
$$
and
$$
|u_\e (\e i_1)- u_\e (\e i_3)|^4=\frac{\e^4}{|v_{13}|^4}\,.
$$
By this, using \eqref{ex:ineq}, we obtain
\begin{eqnarray*}
\e^2\int_{T}|\nabla u(x)|^4\,dx&=&\e^2\int_{T}\frac{1}{|x|^4}\,dx =\frac{\e^2}{2 (\sqrt{2}+1)^4} \int_{T}\Big(\frac{1}{|v_{12}|^4}+\frac{1}{|v_{13}|^4}\Big)\,dx\\&=&
\frac{1}{4 (\sqrt{2}+1)^4} \Big(\frac{\e^4}{|v_{12}|^4}+\frac{\e^4}{|v_{13}|^4}\Big)\\
&=&\frac{1}{4 (\sqrt{2}+1)^4} \Big(|u_\e (\e i_1)- u_\e (\e i_2)|^4+|u_\e (\e i_1)- u_\e (\e i_3)|^4\Big)\,.
\end{eqnarray*}
Summing over all triangles $T$ having vertices in $\Z_\e(\Om)\setminus\{(0,0)\}$ we get
\begin{equation}\label{stima:esempio}
8+(\sqrt{2}+1)^4\e^2\int_{(\Om+B(0,\e))\setminus 2\e W}|\nabla u(x)|^4\,dx \ge \frac 14\sum_{\substack{i,j \in \Z_\e(\Om)\setminus\{(0,0)\}\\|i-j|=1}}|u_\e(\e i)-u_\e (\e j)|^4\,.
\end{equation}
Indeed, every element of the sum appears twice in both sides, since there are two triangles of the interpolation grid having $i$ and $j$ as vertices, with the exception of the $8$ pairs of points on the boundary of $2 \e W$. These are counted only once in the left-hand side, since in this case one of the triangles is contained $2 \e W$. The corresponding element of the sum is simply estimated by $1$.

Now, taking $M>0$ such that $\Om+B(0,\e)\subset B(0, M)$ we have, passing to polar coordinates
\begin{equation*}
\e^2\int_{(\Om+B(0,\e))\setminus 2\e W}|\nabla u(x)|^4\,dx \le \e^2\int_\e^M \frac{1}{r^3}\,dr=\frac 12-\frac{\e^2}{2M^2} \to \frac12
\end{equation*}
as $\e \to 0$. By this, and \eqref{stima:esempio}, \eqref{example:claim} immediately follows.
\end{example}

\subsection{Sobolev scaling - Selection of oscillating states}
In this paragraph we give an example of gradient-type energy finite on non-uniform states as a result of the competition between nearest and next-to-nearest neighbor energies. Such a competition will affect the continuum limit only in the gradient-type scaling leaving the bulk limit unchanged. 
We consider a bounded Borel function $h:[0,1]\to\R$ having a strict absolute minimum $s\in(0,1)$: without loss of generality we may assume that $h(s)=0$.  We now define the following family of energies
\begin{eqnarray*}
E_\e(u)=\sum_{|i-j|=1}\e^2h(|(u_i\cdot u_j)|)+\sum_{|i-j|=\sqrt{2}}\e^2(1-(u_i\cdot u_j)^2).
\end{eqnarray*} 
With similar arguments as in Theorem \ref{theorem:nn} one straightforwardly has that the $\Gamma$-limit of $E_\e(u)$ is given by 
$$
F(Q):=\begin{cases}
         \displaystyle
         4\int_{\Om} h^{++}(\sqrt2 |Q(x)-\tfrac12 I|)\,dx\quad\hbox{if }Q \in L^\infty(\Om; K)\\
         +\infty \quad\hbox{otherwise}
         \end{cases}
$$
where $K$ is defined by \eqref{convex} and $h^{++} \colon[0,1]\to \R$ as in Definition \ref{h+++}. As already observed in Remark \ref{r: h++}(i), since $h^{++}(t)= 0$ for all $t\in[0,s]$, $F$ provides little information on the set of the ground-state configurations. However, the effect of the second term in the energy will be evident in the next-order $\Gamma$-limit, as shown below. Let $F_\e^1:L^{\infty}(\Om;\Mdue)\to[0,+\infty]$ be given by
\begin{eqnarray}\label{F1ennn}
F^1_\e(Q)=
\begin{cases}\displaystyle
\sum_{|i-j|=1}h(|(u_i\cdot u_j)|)+\sum_{|i-j|=\sqrt{2}}(1-(u_i\cdot u_j)^2)&\text{if $ Q\in C_\e(\Om;K)$}\\
+\infty&\text{otherwise}.
\end{cases}
\end{eqnarray}
where we have implicitly used the usual identification $Q(\e i)=Q(u_i)$.

\begin{theorem}\label{thm:oscillating}
Let $F^1_\e$ be as in \eqref{F1ennn} and let $F^1\colon L^\infty(\Om;\Mdue)\to \R \cup\{\infty\}$ be defined as
\begin{eqnarray}\label{Gamma-lim_1}
F^1(Q):=\begin{cases}
         \displaystyle
         \frac{2}{s^2}\int_{\Om} |\nabla Q(x)|^2\,dx\quad\hbox{if }Q \in W^{1,2}(\Om; \partial K_s)\\
         +\infty \quad\hbox{otherwise}
         \end{cases}
\end{eqnarray}         
with $\partial K_s$ as in \eqref{partialK_s}.
Then we have that
\begin{eqnarray}
\Gamma\hbox{-}\liminf_\e F^1_\e(Q)\geq F^1(Q),
\end{eqnarray}
with respect to the weak$^*$ topology of $L^{\infty}(\Om;\Mdue)$.
If in addition there exist $\delta>0$, $C>0$ and $p>1$ such that 
\begin{equation}\label{thm:hregular}
h(x)\leq C|x-s|^{2p}\quad \forall x\in (s-\delta,s+\delta),
\end{equation}
then
\begin{eqnarray}
\Gamma\hbox{-}\lim_\e F^1_\e(Q)=F^1(Q).
\end{eqnarray}\end{theorem}

Before the proof we first state two lemmata.
\begin{lemma}\label{lemmaA}
Let $Q\in W^{1,2}(\Om;\partial K_s)$ and assume that there exist $\widehat v,\widehat w \in W^{1,2}(\Om;S^1)$ such that \eqref{formula_magica_2} holds true. Then 
\begin{eqnarray*}
\frac{1}{s^2}\int_\Om|\nabla Q(x)|^2\ dx=\int_\Om|\nabla \widehat v(x)|^2\ dx+\int_\Om|\nabla \widehat w(x)|^2\ dx.
\end{eqnarray*}
\end{lemma}
\begin{proof}
Take $v,\ w$ as in \eqref{decomposizione}. Then $v,\ w$ also satisfy \eqref{formula_magica_2}. By Proposition \ref{formula_magica} we have that for almost every $x\in\Om$ $\widehat v\otimes \widehat v=v\otimes v$ and $\widehat w\otimes\widehat w=w \otimes w$ or the converse, which implies that 
\begin{eqnarray}\label{lemma:gradienti}
\int_\Om|\nabla \widehat v(x)|^2\ dx+\int_\Om|\nabla \widehat w(x)|^2\ dx=\int_\Om|\nabla v(x)|^2\ dx+\int_\Om|\nabla  w(x)|^2\ dx.
\end{eqnarray}
An explicit computation yelds
\begin{equation}
|\nabla v(x)|^2+|\nabla w(x)|^2=(1+s)|\nabla n(x)|^2+(1-s)|\nabla n^\perp(x)|^2=2|\nabla n(x)|^2.
\end{equation}
On the other hand by \eqref{lifting} and \eqref{fattoregiusto2} we have that 
\begin{eqnarray}
|\nabla Q(x)|^2=s^2|\nabla(n(x)\otimes n(x))|^2=2 s^2|\nabla n(x)|^2= s^2(|\nabla v(x)|^2+|\nabla w(x)|^2)
\end{eqnarray}
hence the conclusion.
\end{proof}

In the statement of the following lemma, as well as in what follows, we use the terminology that a point $i=(i_1,i_2)\in\Z^2$ is called even if $i_1+i_2$ is even, and odd otherwise.

\begin{lemma}\label{lemma:strong_conv2}
Let $\Om$ be an open subset of $\R^2$. Denote by $\widehat W$ the cube obtained by rotating $\sqrt{2}W$ by $\pi/4$. Given a sequence of functions $Q_\e\in C_\e(\Om;K)$, let $Q^{odd}_\e(x)$ and $Q^{even}_\e(x)$ be the odd and the even piecewise-constant interpolations of $Q_\e$ on the cells $\e\{i+\widehat W\}$ for $i$ odd, and $i$ even, respectively. Let $Q^{odd,a}_\e$ be the odd piecewise-affine interpolation of $Q_\e$ having constant gradient on triangles with odd vertices and longest side parallel to $e_1$. Similarly define the even piecewise-affine interpolation $Q^{even,a}_\e$. Then, for all relatively compact $\Om' \subset \subset \Om$
\begin{equation}\label{interpol+}
\int_{\Om'}|Q^{odd}_\e(x)-Q^{odd,a}_\e(x)|^2\,dx \le \e^2 \int_{\Om'+B(0, \sqrt{2}\e)} |\nabla Q^{odd,a}_\e(x)|^2\,dx 
\end{equation}
In particular, if $\int_{\Om'+B(0, \sqrt{2}\e)}|\nabla Q^{odd,a}_\e(x)|^2\,dx$ is bounded uniformly with respect to $\e$ and $\Om'$, then $Q^{odd,a}_\e- Q^{odd}_\e$ converges to $0$ in $L^2(\Om; \Mdue)$ as $\e \to 0$, and each limit point $Q^{odd}$ of $Q^{odd,a}_\e$ belongs to $W^{1,2}(\Om;\partial K)$. The same statement holds with $Q^{even,a}_\e$ and $Q^{even}_\e$ in place of $Q^{odd,a}_\e$, and $Q^{odd}_\e$, respectively.
\end{lemma}

\begin{proof} The proof follows the one of Lemma \ref{Lemma_strong_conv}.
Consider an odd $i\in \Z_\e(\Om')$ and fix $x$ in the interior of $\e \{i +\widehat W\}$. By construction, up to a null set the gradient $\nabla Q^{odd,a}_\e$ is constant on the segment joining $x$ and $\e i$, the only possible exception being  when $x-\e i$ is parallel to one of the vectors $e_1$, $e_1-e_2$, and $e_1+e_2$. By this and the mean value theorem we then have that 
$$
|Q^{odd}_\e(x)-Q^{odd,a}_\e(x)|=|Q^{odd,a}_\e(\e i)-Q^{odd,a}_\e(x)|\le|x-\e i||\nabla Q^{odd,a}_\e(x)|
$$
for a.e.\ $x \in \e \{i+W'\}$. Therefore
$$
|Q^{odd}_\e(x)-Q^{odd,a}_\e(x)|^2 \le \e^2 |\nabla Q^{odd,a}_\e(x)|^2
$$
for a.e.\ $x \in \e \{i+\widehat W\}$. Summing over all such cubes, we get \eqref{interpol+}, and we conclude as in Lemma~\ref{Lemma_strong_conv}.
\end{proof}

\begin{figure}
\begin{center}
\includegraphics[scale=.35 ]{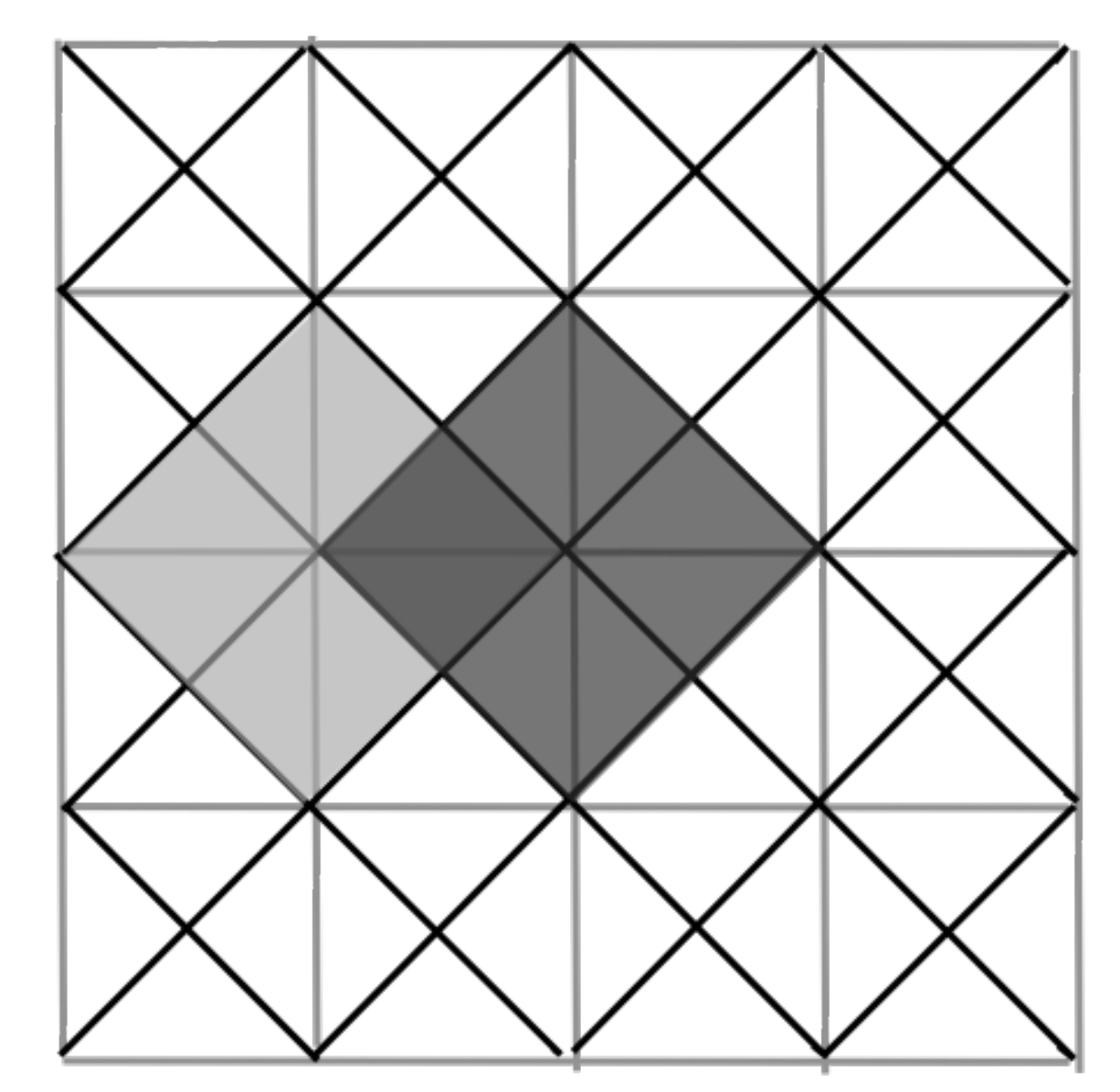}
\caption{Even and odd dual lattices in Theorem \ref{thm:oscillating} and Lemma \ref{lemma:strong_conv2} and their unitary cells in light and dark grey}
\label{fig:dual_lattice}
\end{center}
%\begin{picture}(0,0)
%\put(72,57){$\e$}\put(77,65){\vector(1,0){6}}
%\put(77,65){\vector(-1,0){12}}
%
%\end{picture}
\end{figure}

\begin{proof}[{\it Proof of Theorem {\rm\ref{thm:oscillating}.}}]

Proof of the $\Gamma$-liminf inequality.

Let $Q_\e\in C_\e(\Om;K)$ be such that $Q_\e\wtos Q$ and that $\sup_\e F^1_\e(Q_\e)\leq C$. 
In order to give the optimal lower bound for the energy, we proceed as follows. Consider the rotated cube $\widehat W$, and the odd and even piecewise-affine interpolations $Q_\e^{odd,a}$ and $Q_\e^{even,a}$of $Q_\e$, as in Lemma \ref{lemma:strong_conv2}. If we split $\widehat W$ into two triangles having common boundary in the direction $e_1$, calling $\widehat T_1$ the upper one and $\widehat T_2$ the lower one, by construction $\nabla Q_\e^{odd,a}$ is constant on each cell $\e(\widehat T_1+i)$ and $\e(\widehat T_2+i)$ with $i$ even (so that all the vertices of the cells are odd points). Therefore, also using \eqref{prodotti} and \eqref{modul}, for every $i$ even we have
\begin{eqnarray}\label{identita_gradienti}
\nonumber&&\hskip-2cm
\int_{\e(\widehat T_1+i)}|\nabla Q_\e^{odd,a}(x)|^2\, dx\\
\nonumber &=&\e^2\left(
\frac{|Q_\e(\e(i+e_2)-Q_\e(\e(i-e_1)|}{\sqrt{2}\e}
\right)^2+\e^2\left(
\frac{|Q_\e(\e(i+e_2)-Q_\e(\e(i+e_1)|}{\sqrt{2}\e}\right)^2\\
\nonumber
&=&\frac12\left(|Q_\e(\e(i+e_2))-Q_\e(\e(i-e_1))|^2+|Q_\e(\e(i+e_2))-Q_\e(\e(i+e_1))|^2\right)\\
&=&(1-(u_\e(\e i+\e e_2)\cdot u_\e(\e i-\e e_1))^2)+(1-(u_\e(\e i+\e e_2)\cdot u_\e(\e i+\e e_1)^2)
\end{eqnarray}
Fix $\Om'\subset\subset\Om$. Summing over the even $i\in\Z^2$ we obtain 
\begin{eqnarray*}
&&\hskip-2cm\sum_i\left(\int_{\e(\widehat T_1+i)}|\nabla Q_\e^{odd,a}(x)|^2\ dx+\int_{\e(\widehat T_2+i)}|\nabla Q_\e^{odd,a}(x)|^2\ dx\right)\\
&\leq&2\cdot\frac12\sum_{|i-j|=\sqrt{2},\ i \text{ odd}}(1-(u_\e(\e i)\cdot u_\e(\e j))^2),
\end{eqnarray*}
where the prefactor $2$ accounts for at most two different triangles leading to the same next-to-nearest neighbor interaction, while the other additional prefactor $\frac12$ is due to the passage from an ordered to a non-ordered sum. 
This leads to 
\begin{eqnarray}\label{stimagradiente_odd}
\int_{\Om'}|\nabla Q^{odd,a}_\e(x)|^2\ dx\leq \sum_{|i-j|=\sqrt{2},\ i\text{ odd}}(1-(u_\e(\e i)\cdot u_\e(\e j))^2)\leq F_\e^1(Q_\e)\leq C,
\end{eqnarray}
where $C$ is independent of $\Om'$ and $\e$. This estimate implies in particular that $Q_\e^{odd,a}$ is weakly compact in $W^{1,2}(\Om;\Mdue)$ and thus strongly compact in $L^2(\Om;\Mdue)$. Let $Q^{odd}$ be its limit, by Lemma \ref{lemma:strong_conv2} we have that $Q^{odd}\in W^{1,2}(\Om;\partial K)$. As observed in Remark \ref{remark:approx}, this implies the existence of $\widehat v\in W^{1,2}(\Om;S^1)$ such that $Q^{odd}(x)=\widehat v(x)\otimes \widehat v(x)$. Thus, using \eqref{fattoregiusto2} we get
\begin{equation}\label{stima_odd}
\int_\Om|\nabla \widehat v(x)|^2\ dx=\frac12\int_\Om|\nabla Q^{odd}(x)|^2\ dx\leq \frac12\liminf_\e \sum_{|i-j|=\sqrt{2},\ i\ \text{ odd}}(1-(u_\e(\e i)\cdot u_\e(\e j))^2).
\end{equation}
A similar argument leads to the existence of $\widehat w\in W^{1,2}(\Om;S^1)$ such that $Q_\e^{even,a}\to Q^{even}$, where $Q^{even}(x)=\widehat w(x)\otimes \widehat w(x)$ and
\begin{equation}\label{stima_even}
\int_\Om|\nabla \widehat w(x)|^2\ dx=\frac12\int_\Om|\nabla Q^{even}(x)|^2\ dx\leq \frac12\liminf_\e \sum_{|i-j|=\sqrt{2},\ i\text{ even}}(1-(u_\e(\e i)\cdot u_\e(\e j))^2).
\end{equation}
After summing \eqref{stima_odd} and \eqref{stima_even} we get
\begin{eqnarray}\label{stima_tot}
\Gamma\hbox{-}\liminf_\e F_\e^1(Q)\geq 2\int_\Om|\nabla \widehat v(x)|^2\ dx+ 2\int_\Om|\nabla \widehat w(x)|^2\ dx.
\end{eqnarray}

We now show that
\begin{equation}\label{claim}
Q(x)=\frac12	\widehat v(x)\otimes \widehat v(x)+\frac12 \widehat w(x)\otimes \widehat w(x)\,.
\end{equation}
To do this, we construct the (odd and even) piecewise-constant interpolations $Q_\e^{odd}$ and $Q_\e^{even}$ as in Lemma \ref{lemma:strong_conv2}; by the same Lemma we obtain that 
\begin{eqnarray*}
Q_\e^{odd}\to Q^{odd}, \quad Q_\e^{even}\to Q^{even}.
\end{eqnarray*}
strongly in $L^{2}(\Om;\Mdue)$. Note that $\frac{Q_\e^{odd}+Q_\e^{even}}{2}$ coincides exactly with the piecewise-constant dual interpolation $Q'_\e$ of $Q_\e$ defined in \eqref{rel} (see also Figure \ref{fig:dual_lattice}). The previous discussion and Lemma \ref{easy} imply then that \eqref{claim} holds, and that $Q'_\e \to Q$ strongly in $L^{2}(\Om;\Mdue)$. Using formula \eqref{identity} and arguing as in the previous section, we have that  
\begin{eqnarray*}
\sup_\e\frac{4}{\e^2}\int_{\Om'}h(\sqrt{2}|Q'_\e(x)-\frac12 I|)\ dx\leq \sup_\e\sum_{|i-j|=1}h(|(u_i\cdot u_j)|)\leq C,
\end{eqnarray*}
which, together with the strong compactness of $Q_\e'$ implies that $\int_{\Om}h(\sqrt{2}|Q(x)-\frac12 I|)\ dx=0$. This gives that $Q(x)\in\partial K_s$ for a.e.$x\in\Om$, therefore $Q\in W^{1,2}(\Om; \partial K_s)$. From this, \eqref{stima_tot}, \eqref{claim} and Lemma \ref{lemmaA} the lower-bound inequality follows.

\bigskip

Proof of the $\Gamma\hbox{-}\limsup$ inequality assuming \eqref{thm:hregular}.\\

Let $R>0$ be such that $\Om\subset RW$ and let $Q\in W^{1,2}(RW;\partial K_s)$ denote the (not renamed) extension of $Q\in W^{1,2}(\Om;\partial K_s)$, which exists thanks to the regularity of $\Om$. As observed in Remark \ref{remark:approx} we can suppose that $Q$ is as in \eqref{formula_magica_2} with $v,w\in C^{\infty}(RW;S^1)\cap W^{1,2}(RW;S^1)$. We now define $u_\e$ as
\begin{eqnarray*}
u_\e(\e i):=\begin{cases}
v_i&\text{if $i_1+i_2\in 2\Z$,}\\
w_i&\text{else,}
\end{cases}
\end{eqnarray*}
and $Q_\e\in C_\e(RW;K)$ and $Q_\e'\in C_\e'(RW;K)$ as the piecewise-constant interpolation of $Q(u_\e)$ on the lattices $\Z_\e$ and $\Z_\e'$, respectively. By the regularity of $v$ and $w$ we have that $Q'_\e$ (and thus $Q_\e$ by Lemma \ref{easy}) converge to $Q$ weakly$^*$ in $L^{\infty}(\Om;K)$. Let the piecewise-affine interpolations $Q_\e^{odd,a}$ and $Q_\e^{even,a}$ be defined as in the previous step. By construction, $Q_\e^{odd,a}$ and $Q_\e^{even,a}$ agree with the piecewise-affine interpolations of $w\otimes w$ on the cells of the type $\e(i+\widehat W)$ with $i$ even and with $v\otimes v$ on the cells of the type $\e(i+\widehat W)$ with $i$ odd, respectively. Observe that we have 
\begin{equation}\label{thm:strong_conv}
Q_\e^{odd,a}\to w\otimes w\quad Q_\e^{even,a}\to v\otimes v
\end{equation}
strongly in $W^{1,2}(\Om; \Mdue)$. Using \eqref{identita_gradienti} we deduce that
\begin{equation}\label{limsup_odd}
\sum_{|i-j|=\sqrt{2},\ i \text{ odd}}(1-(u_\e(\e i)\cdot u_\e(\e j))^2)\leq \int_{{\Omega+B(0,\sqrt{2}\e)}}|\nabla Q_\e^{odd,a}(x)|^2\ dx.
\end{equation}
Similarly,
\begin{equation}\label{limsup_even}
\sum_{|i-j|=\sqrt{2},\ i \text{ even}}(1-(u_\e(\e i)\cdot u_\e(\e j))^2)\leq \int_{{\Omega+B(0,\sqrt{2}\e)}}|\nabla Q_\e^{even,a}(x)|^2\ dx.
\end{equation}
Combining \eqref{thm:strong_conv}, \eqref{limsup_odd} and \eqref{limsup_even}, using \eqref{fattoregiusto} and Lemma \ref{lemmaA} we get
\begin{eqnarray}\label{limsup_NNN}\nonumber
&&\hskip-2cm\limsup_\e \sum_{|i-j|=\sqrt{2}}(1-(u_\e(\e i)\cdot u_\e(\e j))^2)\\ \nonumber
&\leq& \int_\Om |\nabla (w(x)\otimes w(x))|^2+|\nabla (v(x)\otimes v(x))|^2\ dx\\&=& 2\left(\int_\Om|\nabla w(x)|^2\ dx+\int_\Om|\nabla v(x)|^2\ dx\right)=F^1(Q).
\end{eqnarray}
%Let us now suppose that \eqref{thm:hregular} holds.  
In view of \eqref{limsup_NNN} it is left to show that 
\begin{equation}\label{claim_2}
\limsup_\e\sum_{|i-j|=1}h(|(u_\e(\e i)\cdot(u_\e(\e j))|)\leq 0.
\end{equation}
Since $Q\in C^{\infty}(RW;\partial K_s)$ and $v,w\in C^{\infty}(RW; S^1)$ we have that 
\begin{equation}\label{thm:s}
|(v(x)\cdot w(x))|=s\quad \hbox{ for all }x\in RW.
\end{equation}
As a result, due to the regularity of $v$ and $w$, for $\e$ small enough we have that 
\begin{equation*}
s-\delta\leq|(u_\e(\e i)\cdot(u_\e(\e j))|\leq s+\delta
\end{equation*}
and 
\begin{equation}\label{thm:lip}
|w(\e i)-w(\e j)|\leq L\e\quad \hbox{ for all } i,j \text{ such that } |i-j|=1. 
\end{equation}
Therefore, by construction of $u_\e$, \eqref{thm:hregular}, \eqref{thm:s} and \eqref{thm:lip} we get that
\begin{eqnarray*}
\sum_{|i-j|=1}h(|(u_\e(\e i)\cdot(u_\e(\e j))|)&\leq& \sum_{|i-j|=1}\Big||(u_\e(\e i)\cdot(u_\e(\e j))|-s\Big|^{2p}\\&=&\sum_{|i-j|=1}\Big||(v(\e i)\cdot(w(\e j))|-s\Big|^{2p}\\&=&\sum_{|i-j|=1}\Big||(v(\e i)\cdot(w(\e j))|-|(v(\e i)\cdot(w(\e i))|\Big|^{2p}\\&\leq&\sum_{|i-j|=1}|w(\e j)-w(\e i)|^{2p}\leq 2L|\Om+B(0,\e)|\e^{2(p-1)}.
\end{eqnarray*}
By this estimate, \eqref{claim_2} follows.
\end{proof}
\begin{remark}
If $h\in C^2(s-\delta,s+\delta)$ for some $\delta>0$ and has a strict minimum at $s$, then $h$ satisfies \eqref{thm:hregular} with $p\geq 1$. In the case when $p=1$ our construction still shows that $\Gamma\hbox{-}\limsup_\e F_\e(Q)<+\infty$ if and only if $Q\in W^{1,2}(\Om;\partial K_s)$.
\end{remark}
\begin{remark}
Note that the prefactor $s^{-2}$ in \eqref{Gamma-lim_1} is proportional to the square of the curvature of $\partial K_s$. Thus, when seen as a function of $s$, we may give to our limit energy a nice interpretation: it quantifies the cost of a unit spatial variation of the order parameter $Q$ depending on its distance from the ordered state. Indeed it is minimal for $s=1$  and it diverges as $s$ goes to zero, where $s=1$ corresponds to a uniform state while $s=0$ corresponds to a disordered state. 
\end{remark}

\subsection{Concentration-type scaling}\label{concentration}
We now consider a different scaling for the functionals $E_\e$ with $h(x)=(1-x^2)$ which will lead to a concentration phenomenon. We define
\begin{equation*}
F_\e^c(Q):=\frac{F_\e(Q)-\inf F_\e}{\e^2|\log \e|}.
\end{equation*}
As usual we may rewrite the functional as
\begin{equation}
F_\e^c(Q)=
\begin{cases}\displaystyle
\frac{1}{|\log \e |}\sum_{|i-j|=1}(1-|(u_i\cdot u_j)|^2)&\text{if $ Q\in C_\e(\Om;K)$}\\
+\infty&\text{otherwise}.
\end{cases}
\end{equation} 
The $\Gamma$-limit of $F_\e^c$ will give rise to concentration phenomena. Following the ideas in \cite{ACXY}, the $\Gamma$-limit of $F_\e(Q)-\inf F_\e$ and of its surface scaling $\frac{F_\e(Q)-\inf F_\e}{\e}$ turn out to be trivial. The scaling we have chosen allows us to consider $F_\e^c$ as a sequence of Ginzburg-Landau type functionals with a non trivial limit. As known in this framework, in order to track the concentration effects one need to define an appropriate notion of convergence of suitable jacobians of the order parameter as well as o notion of degree.
For every $Q:\Om\to K$ we consider the auxiliary vector valued map $A(Q):\Om\to \R^2$ defined as
\begin{equation}
A(Q):=(2Q_{11}-1,2Q_{12}).
\end{equation}
%With the usual identification $\C\simeq\R^2$, we will also regard  as a vector valued map.
Note that $A(Q)$ and $Q$ have the same Sobolev regularity; in particular, if $Q\in W^{1,1}(\Om;K)$ we have that $A(Q)\in W^{1,1}(\Om;\R^2)\cap L^{\infty}(\Om;\R^2)$. Thus, we may define the distributional Jacobian of $A(Q)$ as in \eqref{jwds}. Furthermore if $Q\in W^{1,2}(\Om;K)$ then $A(Q)\in W^{1,2}(\Om;\R^2)\cap L^{\infty}(\Om;\R^2)$ and the following equality holds for almost every $x\in\Om$:
\begin{equation}\label{fattore2}
|\nabla A(Q)(x)|^2=2|\nabla Q(x)|^2.
\end{equation} 
Arguing as in Section \ref{uniform states} and using \eqref{fattore2} we have that
\begin{equation}\label{fattore3}
F^c_\e(Q_\e)=\frac{1}{|\log\e|}\int_{\Om_\e}|\nabla Q^a_\e(x)|^2\,dx=\frac{1}{2|\log\e|}\int_{\Om_\e}|\nabla A(Q^a_\e)(x)|^2\,dx,
\end{equation}
where $Q^a_\e$ is the usual piecewise-affine interpolation of $Q_\e$ on $\Z_\e(\Om)$. 

The following compactness and $\Gamma$-convergence result for $F_\e^c$. In the statement, $\|\cdot\|$ denotes the dual norm of $C^{0,1}_c(\Omega)$.
\begin{theorem}\label{GL-Maintheorem_2d}It holds that:
\begin{itemize}
\item[(i)] {\rm Compactness and lower-bound inequality.}
Let $(Q_\e)$ be a sequence of functions such that $F_\e(Q_\e)\leq C$. Then we can extract
a subsequence (not relabeled) such that, $\| J(A(Q^a_\e))-\pi
\mu\|\to 0$, where $\mu=\sum_{k=1}^m z_k\delta_{x_k}$ for
some $m\in\N$ , $z_k\in\Z$ and $x_k\in\Omega$. Moreover
\begin{eqnarray}\label{Tliminf}
\liminf_\e F_\e^c(Q_\e)\geq \pi|\mu|(\Omega)=\pi\sum_{k=1}^m|z_k|.
\end{eqnarray}
\item[(ii)] {\rm Upper-bound inequality.} Let
$\mu=\sum_{k=1}^mz_k\delta_{x_k}$. Then there exists a sequence
$(Q_\e)$ such that, $\|J(A(Q^a_\e))-\pi \mu\|\to 0$ and
$$
\lim\limits_{\e\to 0} F_\e(Q_\e)=\pi|\mu|(\Omega)=\pi\sum_{k=1}^m|z_k|.
$$
\end{itemize}
\end{theorem}

\begin{proof}
The proof follows by arguing as in \cite{ACXY}, using the identities \eqref{fattore2} and \eqref{fattore3}.
\end{proof}

\begin{figure}
\begin{center}
\includegraphics[scale=.35]{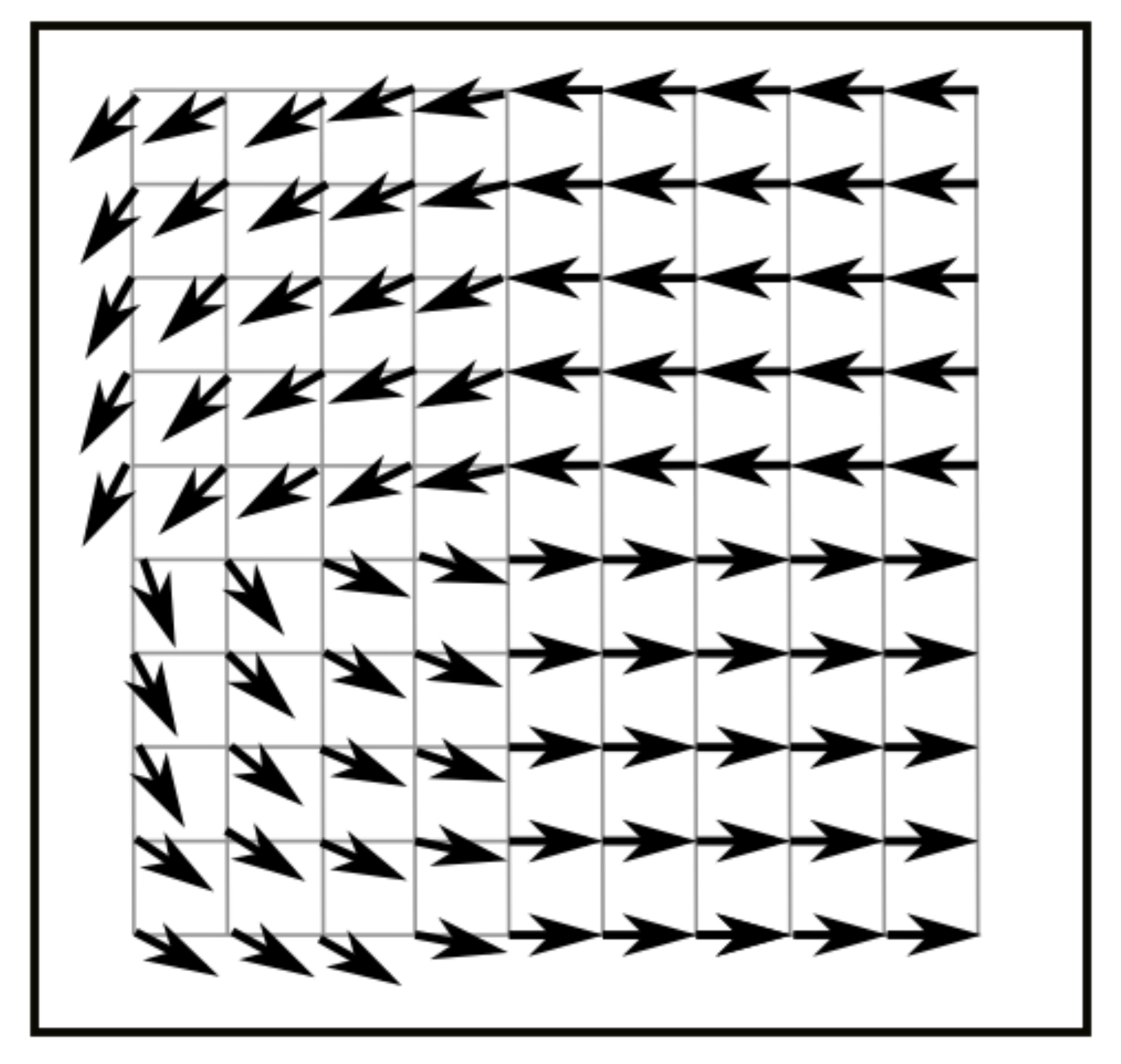}
\caption{Half charged discrete vortex for $\mu=\delta_{0}$}\label{fig_vortex}
\end{center}
%\begin{picture}(0,0)
%\put(72,57){$\e$}\put(77,65){\vector(1,0){6}}
%\put(77,65){\vector(-1,0){12}}
%
%\end{picture}
\end{figure}

\begin{remark}
Coming  back to the energy description in terms of the order parameter $u\in S^1$, the result above has an interesting interpretation. Assume that $0\in\Om$ and take $\mu=\delta_{0}$. By the previous theorem there exists a recovery sequence $Q_\e$ such that $J(A(Q_\e^a))$ approximates $\pi\mu$ in the dual norm of $C^{0,1}_c(\Omega)$. Following the same ideas as in \cite{ACXY}, such $Q_\e$ can be obtained by discretizing the map $Q(x)=\bigl(\frac{x}{|x|}\bigr)^{\frac12}\otimes\bigl(\frac{x}{|x|}\bigr)^{\frac12}$ (here, the square root is meant in the complex sense). The presence of the square root is easily explained: for $u(x)=\bigl(\frac{x}{|x|}\bigr)^\frac12$ we have $J(u)=\frac{\pi}{2}\delta_0$. Moreover, for every $Q=Q(u)$ we have that $A(Q)=(u_1+i u_2)^2$ which in turn implies $J(A(Q))=2J(u)=\pi\delta_0=\pi\mu$. Since $Q_\e=Q(u_\e)$, with $u_\e$ the discretization of $u(x)$, read in terms of the vectorial order parameter, the optimal sequence is pictured in Figure \ref{fig_vortex} and the energy 
concentrates on the topological singularity of a map having half degree. Of course, by the locality of the construction of the recovery sequence in \cite{ACXY}, this observation extends to any $\mu$ of the type $\mu=\sum_{k=1}^m z_k\delta_{x_k}$ for some $m\in\N$ , $z_k\in\Z$ and $x_k\in\Omega$, thus asserting that the optimal sequence in therms of $u$ looks like a complex product of a finite number of maps with half-integer singularities.  

%Up to small error vanishing in the limit, 
%Moreover, for $\e$ small enough, up to a small error, the Jacobian of the map $A(Q_\e^a)$ and of its projection on $S^1$ that we denote by $w_\e:=\frac{A(Q_\e^a)}{|A(Q_\e^a|}$ are very close to each other. Since
%\begin{equation}
%\pi{\rm deg}(w_\e,x):= J(w_\e).
%\end{equation} 
%In this sense we may interpret the above result as stating that in the limit the degree of $w_\e$, that is the degree of $A(Q_\e^a)$ has to be an integer number. Note also that any limit point $Q$ of $Q_\e^a$ is such that $Q(x)\in\partial K$ for almost every $x\in\Om$. This implies that there exists $u\in S^1$ such that $Q=Q(u)$. For such a $Q$ we have that $A(Q)=(u_1+i u_2)^2$. Assuming enough regularity on $u$ (this is not guaranteed by the compactness result) the last formula eventually implies that ${\rm deg } u=\frac{1}{2}{\deg A(Q) }\in\frac{1}{2}\Z$.  
\end{remark}

\begin{remark}\label{rem:concentration-longrange}
The Lebwohl-Lasher model we have considered in this section belongs to a more general class of two-dimensional Maier-Saupe models with long-range interactions, which in the bulk scaling can be written as 
\begin{eqnarray*}
E_\e(u)=\sum_{\xi\in \zN }\sum_{\alpha\in R_\e^\xi(\Omega)} \e^2c^\xi (1-(u(\e\alpha),u(\e\al+\e\xi)^2),
\end{eqnarray*}
with $c^\xi=c^{\xi\perp}$, $c^{e_1}>0$ and such that $\sum_\xi|\xi|^2c^\xi<+\infty$. A relevant example of energy models falling into this class has been proposed in \cite{Silvano} where $c^\xi=|\xi|^{-6}$. The authors are indeed interested in models where a particle has a large number of interactions, this being in spirit closer to the mean-field approach of the Maier-Saupe theory.  As a consequence of Theorem 5 in \cite{ACXY}, the results stated in Theorem \ref{GL-Maintheorem_2d} continue to hold for this class of functionals provided we replace the prefactor $\pi$ by $\pi\sum_\xi|\xi|^2c^\xi$. We also observe, as a byproduct of this result, that bulk and surface-type scalings of $E_\e$ turn out to have trivial $\Gamma$-limits. 
\end{remark}

\bibliography{BCS-arXiv}
\bibliographystyle{plain}

\end{document}